\documentclass[11pt]{amsart}

\usepackage[left=1.5in, right=1.5in]{geometry}
\usepackage[latin1]{inputenc}
\usepackage{amsfonts}
\usepackage[toc,page,title,titletoc,header]{appendix}
\usepackage{graphicx,psfrag,epsfig,multirow}
\usepackage{amssymb,amsmath,amscd,amsthm,verbatim}
\numberwithin{equation}{section}
\newtheorem{Thm}{Theorem}[section]
\newtheorem{Def}{Definition}[section]
\newtheorem{Lm}{Lemma}[section]
\newtheorem{Prop}[Lm]{Proposition}
\newtheorem{Rk}{Remark}[section]
\newtheorem{Cor}[Lm]{Corollary}

\def\bdef{\begin{Def}}
\def\endef{\end{Def}}
\def\bthm{\begin{Thm}}
\def\ethm{\end{Thm}}
\def\bprop{\begin{Prop}}
\def\enprop{\end{Prop}}
\def\blm{\begin{Lm}}
\def\elm{\end{Lm}}
\def\bcor{\begin{Cor}}
\def\ecor{\end{Cor}}
\def\brm{\begin{Rem}}
\def\erm{\end{Rem}}
\def\bfig{\begin{picture}}
\def\efig{\end{picture}}
\def\be{\begin{eqnarray}}
\def\ee{\end{eqnarray}}
\def\beal{\begin{aligned}}
\def\enal{\end{aligned}}
\def\beq{\begin{equation}}
\def\eeq{\end{equation}}

\def\Z{\mathbb Z}

\def\C{\mathbb C}
\def\N{\mathbb N}

\def\R{\mathbb R}
\def\eps{\varepsilon}
\def\al{\alpha}
\def\bt{\beta}
\def\dt{\delta}

\def\tt{\theta}

\def\Sm{\Sigma}
\def\sm{\sigma}

\def\om{\omega}
\def\pt{\partial}

\title{Continuous averaging proof of the Nekhoroshev theorem}

\author{Jinxin Xue\\ University of Maryland-College Park}
\thanks{Email: jxue@math.umd.edu }

\begin{document}
\maketitle
\begin{abstract}
In this paper we develop the \textit{continuous averaging} method of Treschev to work on the simultaneous Diophantine approximation and apply the result to give a new proof of the Nekhoroshev theorem. We obtain a sharp normal form theorem and an explicit estimate of the stability constants appearing in the Nekhoroshev theorem.
\end{abstract}
\section{Introduction}In the papers \cite{Tr1,Tr2}, Treschev developed a new averaging method called continuous averaging. It is a powerful tool to derive sharp constants in the exponentially small splitting problems in Hamiltonian systems with one and a half degrees of freedom. But the technicality becomes very heavy when we use the method to study Hamiltonian systems of more degrees of freedom. For this reason, the method has not been applied to other problems yet.

In this paper, we use the continuous averaging to give a new proof of the Nekhoroshev theorem.
We consider the following analytic nearly integrable Hamiltonian system:
 \begin{equation}H(I,\tt,x,y)=H_0(I)+\eps H_1(I,\tt,x,y),\label{eq: main}\end{equation}
The phase space is \[(I,\tt,x,y)\in \mathcal{D}:=\mathcal{G}^n\times(\R/2\pi\Z)^n\times \mathcal{W}^{2m}\subset \R^n\times(\R/2\pi\Z)^n\times \R^{2m},\quad n\geq 2,\ m\geq 0.\]
We complexify the variables and extend the domain of $(I,x,y)$ to a $\sm$ neighborhood and that of $\tt$ to a $\rho$ neighborhood of the original domains respectively.
The extended phase space to the complex domain is \[\mathcal{D}(\rho,\sm):=(\mathcal{G}^n+\sm)\times\left((\R/2\pi\Z)^n+\rho\right)\times (\mathcal{W}^{2m}+\sm)\subset\C^n\times(\C/2\pi\Z)^n\times \C^{2m},\]
 where $\rho$ is the width of analyticity in $\tt$ and $\sm$ is that of the slow variables $I,x,y$.

As stated in \cite{Ne,L1,L2,LN,LNN,Po,BM}, Nekhoroshev Theorem ensures that when the unperturbed Hamiltonian $H_0$ is quasi-convex, by which we mean that the set $\{I|\ H_0(I) \leq E\}$ is strictly convex, the following general estimate holds for sufficiently small $\eps$:\begin{equation}\Vert I(t)-I(0)\Vert\leq C_0\eps^{b},\quad \textrm{when}\quad |t|\leq \mathcal{T}=C_1e^{C_2/\eps^a}.\label{eq: ccc}\end{equation} for some constants $a,b,C_0,C_1,C_2>0$ independent of $\epsilon$,
where $I(t)$ is the action variable component of any orbit associated to Hamiltonian~(\ref{eq: main}) with initial condition in the set $\mathcal{D}$.

There are many works studying the stability exponents $a$ and $b$ (c.f.\cite{LN,Po,BM}). Their approaches are based on a careful study of the geometric and number theoretical aspects of resonances. Instead, in this paper we try to sharpen the estimates in the analytic part of the proof using continuous averaging to obtain an improved normal form (see Theorem~\ref{Thm: main}). Then we apply the normal form to Lochak's argument to get the Nekhoroshev theorem (see Theorem~\ref{Thm: global}) where all the stability constants are estimated explicitly. In this paper, we only work on the case $a=b=1/2n$. But the normal form theorem can be easily applied to other prescribed $a$ and $b$ to get the corresponding $C_2$.

The method of Lochak is called the simultaneous Diophantine approximation, which turns out to be an important alternative to the classical approach via small divisor techniques, as explained in \cite{L2}. Its main idea is to do the averaging in a vicinity of a periodic orbit. So it is essentially an averaging procedure for systems with one fast angle. In general, we can kill the dependence on the fast angle up to exponential smallness. This makes the simultaneous Diophantine approximation suitable to prove the Nekhoroshev theorem.
The work \cite{PT} can be considered as a development of the continuous averaging to the small divisor case. In this paper, it is the first time that the continuous averaging has been developed to the simultaneous Diophantine approximation.

We point out the relation between continuous averaging and some important PDEs. The idea of the continuous averaging is to study the averaging procedure using PDE instead of iterations. The PDE has the form $H_\dt=\{H,F\}$, where $F$ is the Hilbert transform of $H$ in some special cases (see Section 3 for more details). This type of equation has been studied (c.f. \cite{CCF}) as a simplified model for quasi-geostrophic equation (c.f. \cite{KNV}), incompressible Euler equation, etc. It would be interesting if we could apply some PDE techniques to our problem.

To state our theorems, we need the following definitions.
\begin{Def}\begin{enumerate}
\item We use $|\cdot|,\ |\cdot|_2,\ |\cdot|_\infty$ to denote the $l_{1},\ l_2,\ l_\infty$ norms for a vector in $\R^n$ or $\Z^n$. Without causing confusion, we also use $|\cdot|$ to denote the absolute value of a function whose range is in $\R$ or $\C$.
\item For a function $f(I,\tt,x,y)$, the $weighted\ Fourier\ norm$ is defined as:\[\Vert f\Vert_{\rho'}=\sup_{I,x,y}\sum_{k\in\Z^n}|f^k|e^{|k|\rho'},\ \rho'\leq\rho,\]
if we have the Fourier expansion $f(I,\tt,x,y)=\sum_{k\in\Z^n}f^k(I,x,y) e^{i\langle k,\tt\rangle}$, 
and the other variables $(I,x,y)\in (\mathcal{G}^n+\sm)\times (\mathcal{W}^{2m}+\sm)$.
\item $\max\{\Vert H_1\Vert_\rho,\Vert\nabla H_1\Vert_\rho\}:=\mu$.
\end{enumerate}\label{def: norm}
\end{Def}
We also use the following definition to characterize the convexity of the unperturbed part $H_0(I)$.
\begin{Def}
Consider a Hamiltonian $H_0(I)$ defined on $G^n+\sigma$. Then, we define the associated constants $M_\pm>0$ to characterize the convexity of $H_0(I)$.
\begin{equation}\begin{aligned}
&0<M_- |v|_2^2\leq |\langle\nabla^2 H_0(I) v,v\rangle|,\quad v\in\R^n\setminus\{0\},\quad I\in \mathcal{G}^n.\\
&|\nabla^2 H_0(I) v |_2\leq M_+| v|_2,\quad v\in \C^n\setminus\{0\},\quad I\in \mathcal{G}^n+\sigma. 
\end{aligned}\label{eq: ineq}
\end{equation}
\label{def: convex}
\end{Def}
Now we state a simplified version of our main theorem. The complete version is stated in the next section. 
\begin{Thm}\label{Thm: simple}
Consider a Hamiltonian system~$(\ref{eq: main})$ satisfying inequalities~(\ref{eq: ineq}) in Definition~\ref{def: convex} and $n\geq 2, m\geq 0$. For every orbit $(I,\tt,x,y)(t)$ with initial condition $(I,\tt,x,y)(0)\in \mathcal{D}(\rho,\sigma)$ and $(x(t),y(t))\in \mathcal{W}^{2m}+\sm$,
we have the following estimates provided $\eps$ is small enough.
\[| I(t)-I(0)|_2\leq \dfrac{8\sqrt{n -1}M_{+}}{M_-^2}\eps^{1/2n}|\nabla H_0|_\infty,\]
for\[|t|\leq \dfrac{1}{|\nabla H_0|_\infty}\exp\left(\left(\dfrac{M_-}{M_+}\right)^{2}\dfrac{\rho_1}{8\sqrt{n-1}\eps^{1/2n}}\right),\]
where $\rho=\rho_1+2\rho_2+\rho_3$ satisfies
\begin{itemize}
\item $\dfrac{6\mu}{\rho_2^n} \leq \dfrac{M_+^{2}}{M_-^{3}} |\nabla H_0|_\infty^2$,
\item $\dfrac{3e^{\sm}\rho_3}{n!|\nabla H_0|_\infty}\left(\dfrac{\rho_1M_-^2}{\rho_3 4\sqrt{n-1}M_+}\right)^{n+1}<\dfrac{4\sm}{25}$.\\
\end{itemize}
The norm $|\cdot|_\infty$ is taken over $I\in\mathcal{G}^n$.
\end{Thm}
This theorem gives the estimate of the stability constant $C_2$ in \eqref{eq: ccc}. For a given system, we need to optimize $\rho_1$ under the constraints in the theorem. We see that the decomposition of $\rho$ can be qualitatively written as $\rho=\rho_1+c_0 \mu^{1/n}+c_1\rho_1^{1+1/n},$ where the constants $c_0=c_0(M_\pm,|\nabla H_0|_\infty)$ and $c_1=c_1(M_\pm,|\nabla H_0|_\infty,n,\sigma)$. Though not solved explicitly, we expect our estimate here improves the previous results \cite{LNN,N} since the continuous averaging method gives us an improved normal form theorem (see Theorem~\ref{Thm: main}).   

A possible application of the result is to the 3-body problem in order to get long time stabilities. This direction is already pioneered in \cite{N}. But the mass ratio of Jupiter to the sun obtained in \cite{N} is too small to be satisfactory. On the other hand, in \cite{FGKR}, the authors construct diffusing orbits for restricted planar 3-body problem. The diffusion time there is polynomial w.r.t. $1/\eps$. 

The paper is organized as follows. First we give a complete statement of the main theorem and compare it with previous results in Section~\ref{section: complete}. Then we state a normal form Theorem~\ref{Thm: main} about averaging in a vicinity of periodic orbits in Section~\ref{section: normal}. This is the main result that we obtain using continuous averaging, which improves the corresponding one in \cite{LNN,N}.  Then we give a brief introduction to the continuous averaging method in Section~\ref{section: ca}.  After that we give a proof of Theorem~\ref{Thm: main} in Section~\ref{section: main}. This section is a higher dimensional generalization of the case studied in Section~\ref{section: ca}. We try to draw analogy between the two sections. With the normal form theorem, we first show local stability result of Nekhoroshev theorem in Section~\ref{section: local}, and then global stability in Section~\ref{section: global}. Here local stability means the stability result in a neighborhood of a periodic orbit and global stability means stability for all initial conditions. Finally, we have two appendices~\ref{section: appendix} and~\ref{section: majorant}. The first one contains some technical estimates for the continuous averaging. The second one is some basics of majorant estimates.
\section{The complete statement of the main theorem and discussions}\label{section: complete}
We give a complete statement of the main theorem as follows.
\begin{Thm}\label{Thm: global}
Under the same assumption as Theorem~\ref{Thm: simple}, we have \[| I(t)-I(0)|_2\leq \dfrac{8\sqrt{n -1}M_{+}}{M_-^2}\eps^{1/2n}|\nabla H_0|_\infty,\]
\[\mathrm{for\ }|t|\leq \mathcal{T}=\dfrac{1}{|\nabla H_0|_\infty}\exp\left(\left(\dfrac{M_-}{M_+}\right)^{2}\dfrac{\rho_1}{8\sqrt{n-1}\eps^{1/2n}}\right),\]
where $\rho=\rho_1+2\rho_2+\rho_3$, provided the following restrictions are satisfied.
\begin{itemize}
\item $\dfrac{6\mu}{\rho_2^n} \leq \dfrac{M_+^{2}}{M_-^{3}} |\nabla H_0|_\infty^2$,
\item
$\eps^{1/2n}\leq \min\Big\{\dfrac{\sqrt{n-1}|\nabla H_0|^2_\infty}{5n\mu M_-}(\rho_2+\rho_3)^n,\  \dfrac{M_-^2}{4\sqrt{n-1}|\nabla^3 H_0|_\infty},$\\
$\dfrac{M_-^2}{8\sqrt{n-1}M_+}\left(\dfrac{\sm}{5(\sqrt{n}+2\sqrt{m})|\nabla H_0|_\infty },\  \dfrac{\rho_1}{2(n+2m)\rho_3},\ \dfrac{5\rho_1^2}{2\sm\rho_3}
    \right)\Big\}.$
\item $\dfrac{3e^{\sm}\rho_3}{n!|\nabla H_0|_\infty}\left(\dfrac{\rho_1M_-^2}{\rho_3 4\sqrt{n-1}M_+}\right)^{n+1}<\dfrac{4\sm}{25}$,\\
    $\eps^{1/2n}\dfrac{16n\sqrt{n-1}M_+}{\rho_1 M_-^2}\left(1+\ln\dfrac{\rho_1^2M_-^4\eps^{-1/n}}{32n(n-1)M_+^2\rho_3}\right)\leq 1.$
\end{itemize}
The norm $|\cdot|_\infty$ is taken over $I\in\mathcal{G}^n$.
\end{Thm}

The constant $\mu$ plays the same role as the constant $E$ in \cite{LNN,N}. It is dual to $\eps$ since only the product $\eps\mu$ enters the original Hamiltonian. We need the smallness of $\mu$ to make the first bullet point in Theorem~\ref{Thm: global} satisfied. The same restriction is expressed in \cite{LNN,N} by introducing a constant $g$. The second bullet point can be satisfied easily by taking $\eps$ small enough. To improve the stability time, we want $\rho_1$ to be as large as possible, but the third bullet point gives a restriction of $\rho_1$ so that we need to optimize among $\rho_1,\rho_2,\rho_3$. This restriction appears due to the finiteness of the width of analyticity of the action variables $I$ and degenerate variables $x,y$. It shows up in a different form in \cite{LNN} as item $(ii)$ of Theorem 2.1, where the choice of $R$ there can be as small as $\eps^{1/2n}$. We will give more discussions in Remark~\ref{Rk: rho} and~\ref{Rk: global}. We will see from the following Theorem~\ref{Thm: main} that our normal form theorem obtained from continuous averaging improves that obtained from the iteration method. Therefore we see we also get improved $C_2$ here even though $\rho_1$ is not expressed explicitly. 

\section{Normal form}\label{section: normal}
Our main work in this paper is to obtain a normal form theorem using continuous averaging. Following Lochak, we do the averaging in a neighborhood of a periodic orbit.
\begin{Def}
We define $\om^*=(p_1,p_2,\cdots,p_n)/\bar{T},\quad p_i\in\Z,\ \bar{T}\in\R\setminus\{0\}$, and $\mathrm{g.c.d.}(p_1,p_2,\cdots,p_n)=1$. This is the frequency vector of a periodic orbit of the unperturbed Hamiltonian $H_0$. The period $T$ of this vector is, $T=2\pi\bar{T}$.
\end{Def}
Integer vectors $k$ with $\langle\om^*,k\rangle\neq 0$ give us \begin{equation}|\langle k,\om^*\rangle|=|k\cdot(p_1,p_2,\cdots,p_n)|/\bar{T}\geq 1/\bar{T}\label{eq: tbar}\end{equation}

After a proper translation in the space of action variables, wlog, we assume \[\om(0):=\dfrac{\partial H_0}{\partial I}(0)=\om^*.\]
We can split the Hamiltonian~(\ref{eq: main}) into four parts
\begin{equation}H(I,\tt,x,y)=\langle\om^*,I\rangle+G(I)+\eps\bar{H}(I,\tt,x,y)+\eps\tilde{H}(I,\tt,x,y).\label{eq: hamiltonian}\end{equation}
where each of the terms is given in the next definition.
\begin{Def}
We use the Taylor expansion of $H_0$ to split it as
 \[H_0(I)=\langle\om^*,I\rangle+G(I),\]
where $G(I)$ contains the higher order terms.
For $H_1$ part, we use the Fourier expansion $\displaystyle H_1=\sum_{k\in\Z^n}H^k(I,x,y)e^{i\langle k,\tt\rangle}$ to write
\[\eps H_1(I,\tt,x,y)=\eps\bar{H}(I,\tt,x,y)+\eps\tilde{H}(I,\tt,x,y),\]
where \begin{equation}
\begin{aligned}\nonumber
&\eps\bar{H}:=\eps\sum_{\langle k,\om^*\rangle=0}H^k e^{i\langle k,\tt\rangle},\quad \textrm{the resonant part,}\\
&\eps\tilde{H}:=\eps\sum_{\langle k,\om^*\rangle\neq 0}H^k e^{i\langle k,\tt\rangle}, \quad \textrm{the nonresonant part.}\\
\end{aligned}\nonumber
\end{equation}\label{def: resonant}
\end{Def}
The Hamiltonian equations can be written as
\begin{equation}
\begin{aligned}
&\dot{I}=-\eps\dfrac{\pt\bar{H}}{\pt \tt}(I,\tt,x,y)-\eps\dfrac{\pt\tilde{H}}{\pt \tt}(I,\tt,x,y),\\
&\dot{\tt}=\om^*+\nabla G(I )+\eps\dfrac{\pt\bar{H}}{\pt I}(I,\tt,x,y)+\eps\dfrac{\pt\tilde{H}}{\pt I}(I,\tt,x,y),\\
&\dot{x}=-\eps\dfrac{\pt H_1}{\pt y}, \quad \dot{y}=\eps\dfrac{\pt H_1}{\pt x}.\\
\end{aligned}\label{eq: Hamiltonian eq}
\end{equation}

\begin{Thm}
Suppose $|I|_2,|x|_2,|y|_2 \leq \mathcal{R},\ (I,\tt,x,y)\in \mathcal{D}(\rho,\sm)$ for some $\mathcal R, \rho,\sigma>0$. Then
there exists $\eps_0>0$, such that for any $0<\eps<\eps_0$, there exist $\rho_1,\rho_2,\rho_3>0$ such that $\rho_1+2\rho_2+\rho_3=\rho$ and a symplectic change of variables,
$(I,\tt,x,y)\rightarrow(I',\tt',x',y')$, for $| I'|_2 \leq \mathcal{R}$, $(I',\tt',x',y')\in \mathcal D(\rho_2,4\sm/5)$, which sends the Hamiltonian~$(\ref{eq: hamiltonian})$ to the following normal form:
\[H=\langle\om^*,I'\rangle+G(I')+\eps \bar{\Psi}(I',\tt',x',y')+\eps\tilde{\Psi}(I',\tt',x',y').\]
with the nonresonant part $($see Definition~\ref{def: resonant}$)$\[\left\Vert\tilde{\Psi}(I',\tt',x,y)\right\Vert_{\rho_2}\leq \dfrac{5\mu}{\rho_2^n} \exp\left(-\dfrac{2\pi\rho_1}{M_+\mathcal{R}T}\right),\]
the resonant part $\left\Vert\bar{\Psi}\right\Vert_{\rho_2}\leq \dfrac{5\mu}{\rho_2^n},$ and the change of variables
\[| (I',\tt',x',y')-(I,\tt,x,y)|_\infty\leq \dfrac{5\eps\mu T}{2\pi(\rho_2+\rho_3)^n}.\]
The $\eps_0$, $\mathcal{R}$, $K=\dfrac{2\pi\rho_1}{\rho_3M_+\mathcal{R}T}$ and $\rho_1,\rho_2,\rho_3$ satisfy the following restrictions,
\begin{itemize}
\item $5(\sqrt{n}+2\sqrt{m})\mathcal{R}<\sm,$
\item $2(n+2m)\leq K,\quad 2\sm/(5\rho_1)\leq K$,
\item $\dfrac{3}{\pi}e^\sm\eps\mu T\dfrac{(2K)^{n-1}K^2\rho_3}{n!}\left(1+\frac{2n}{\rho_3K}\left(1+\ln\dfrac{2K^2\rho_3}{n}\right)\right)\leq\dfrac{4\sm}{25}$.
\end{itemize}
Moreover, $\rho_1$ can be arbitrarily close to $\rho$ if $\eps$ is sufficiently small (see Remark~\ref{Rk: rho}).
\label{Thm: main}
\end{Thm}
The exponential smallness obtained here improves that of \cite{LNN, N}. Continuous averaging enables us to get rid of some extraneous numerical factors that worsen the estimates. Moreover, our method has an advantage, that is we do not need to do a preliminary transform which is necessary in \cite{LNN, LN}. The proof of this result is contained in Section~\ref{section: main}.
\section{A brief introduction to the continuous averaging}\label{section: ca}
In this section, we give an introduction to the continuous averaging method. Please see the chapter 5 of \cite{TZ} for more details.
We try to explain the key points of the method that will be used in our later proof. 
\subsection{Derivation of the continuous averaging equation}\label{subsection: derivation}
We write the Hamiltonian~(\ref{eq: main}) as $H(z),\ z=(I,\tt,x,y)$. Suppose we have a symplectic change of variables $z=z(Z(\dt),\dt)$ depending on parameter $\dt$, where $Z(\dt)$ denotes the new variables. Then we have
\begin{equation}
H(z)=H(z(Z,\dt)):=H(Z,\dt).\nonumber
\end{equation}
\begin{equation}
\begin{aligned}
& \dfrac{\pt H}{\pt Z}\dfrac{dZ}{d\dt}+\dfrac{\pt H}{\pt\dt}=0.
\end{aligned}
\label{eq: DF}
\end{equation}
If we choose the flow of $\dfrac{dZ}{d\dt}$ to be the Hamiltonian flow generated by a Hamiltonian isotopy $F(Z,\dt)$, i.e.\\
\begin{equation}\dfrac{dZ}{d\dt}=JdF(Z,\dt),\quad
J=\left[\begin{array}{cccc}
0 &   -\mathrm{Id}      \\
\mathrm{Id} &   0     \\
\end{array}\right],\label{eq: JdF}\end{equation}
 with initial value $Z\big|_{\dt=0}=z$, then the change of variables is symplectic and we get
\begin{equation}
H_\dt=-\{H,F\}_Z=-\{H,F\}_z,
\label{eq: cont}
\end{equation}
where the subscript $\dt$ means partial derivative. The last equality follows from the fact that the Poisson bracket is invariant under symplectic transformations. In the following, we only work with the variables $z$.\\
To simplify our discussion, we consider a special case of~(\ref{eq: main}) with $m=n=1$. A further simplification is to consider only time-periodic nonautonomous systems. This is equivalent to requiring that $H_0(I)=I$ in equation~(\ref{eq: main}) and $H_1(x,y,\tt)$ independent of $I$. From equation~(\ref{eq: cont}), we have:
\begin{equation}
H_{\dt}=-F_\tt -\{ H,F\}_{(x,y)},
\label{eq: nonauto}
\end{equation}
where $\{\cdot,\cdot\}_{(x,y)}$ stands for the $x,y$ part of the Poisson bracket.\\
 Our goal is to show that

 \emph{if we choose a suitable Hamiltonian isotopy $F$ and extend $\delta$ as large as possible, the dependence of $H$ on $\tt$ can be killed to be exponentially small, i.e. $O(e^{-c/\eps})$ for some constant $c$.}

 Suppose $H(z,\dt)$ has Fourier expansion \begin{equation}\displaystyle H(z,\dt)=I+\eps\langle H_1\rangle+\eps\sum_{k\in\Z\setminus \{0\}}H^{k}(x,y,\dt)e^{ik\tt},\label{eq: egham}\end{equation}
where $\eps\langle H_1\rangle$ means the zeroth Fourier coefficient of $H_1$.\\
We choose the Hamiltonian isotopy $F$ as the ``Hilbert transform":
\begin{equation}
\displaystyle F(z,\dt)=-\sum_{k\in\Z\setminus\{0\}}i\eps\sm_k H^{k}(x,y,\dt)e^{ik\tt},\quad \sm_k=\mathrm{sgn} (k).\label{eq: hilbert}
\end{equation}
Now equation~(\ref{eq: nonauto}) has the form in terms of Fourier coefficients:
\begin{equation}
H_{\dt}^k=-|k|H^k+i\eps\sm_k\{\langle H_1\rangle,H^k\}_{(x,y)}+i\eps\sum_{l+m=k}\sm_m\{H^l,H^m\}_{(x,y)},\ k\in\Z\setminus\{0\}.
\label{eq: fourier}
\end{equation}
We show this $F$ is the good choice that makes the dependence on $\theta$ decrease exponentially.
\subsection{The choice of the Hamiltonian isotopy $F$}
\subsubsection{Heuristic argument}
Following \cite{TZ}, we explain here the heuristic ideas that justify this choice of $F$. If we set $\eps=0$ in~(\ref{eq: fourier}), we get
\[H_\dt^k=-|k|H^k,\]
whose solution tends to zero as $\dt\to\infty$. If we neglect the third term in the RHS of~(\ref{eq: fourier}), we have
\[H_\dt^k=-|k|H^k-i\eps\sm_k\{H^k,\langle H_1\rangle\}_{(x,y)}.\]
It has an exact solution of the form
\begin{equation}
H^k(I,x,y,\dt)=e^{-|k|\dt}H^k(I,x,y,0)\circ g^{-i\eps \sm_k \dt},
\label{eq: exact}
\end{equation}
where $g$ means the Hamiltonian flow generated by the Hamiltonian $\langle H_1\rangle$. Notice the imaginary unit $i$ here. It tells us that the flow is considered with purely imaginary time. As $\dt$ increases, the complex width of analyticity is lost gradually. So formula~(\ref{eq: exact}) has sense only if we take $\eps\dt<\rho$, where $\rho$ is the width of analyticity in $\tt$. This is an obstacle for the extendability of $\dt$.

We see from the heuristic argument that this choice of $F$ gives us the exponential decay as well as a good guess for the stopping time.
\subsubsection{Comparison with the Lie method}\label{subsection: comparison}
The Lie method is used in the works \cite{N,LN,LNN}. Before working out the detailed proof of the above heuristic argument, we explain the ``Hilbert transform" first. In fact this choice of $F$ is strongly related to the classical averaging theory. Let us recall what we usually do in the Lie method.

Define the linear operator of taking Lie derivative along the Hamiltonian flow generated by the Hamiltonian function $\hat{F}$: $\mathcal{L}_{\hat{F}} H =\{H,\hat{F}\}$.

The time-$1$ map of~(\ref{eq: DF}) and~(\ref{eq: cont}) is:
\begin{equation}
\begin{aligned}
& H\big|_{\dt=1}=e^{\mathcal{L}_{\hat{F}}}H\\
& =H+\{H,\hat{F}\}+\dfrac{1}{2}\{\{H,\hat{F}\},\hat{F}\}+\cdots\\
& =H_0+\eps H_1 +\{H_0,\hat{F}\}+h.o.t.
\end{aligned}
\nonumber
\end{equation}
In each step of iteration, we need to solve the cohomological equation \begin{equation}\eps H_1+\{H_0,\hat{F}\}=0.\label{eq: cohom}\end{equation}
In fact, we are only able to solve \begin{equation}\eps H_1-\eps \langle H_1\rangle+\{H_0,\hat{F}\}=0.\label{eq: cohom1}\end{equation}
By comparing the Fourier coefficients, we obtain the following 
\begin{equation}
\begin{aligned}
&\eps H^k(I)+ik\hat{F}^k=0\Longrightarrow \hat{F}^k=i\dfrac{\eps H^k(I)}{k},\quad k\neq 0.\\
\end{aligned}
\label{eq: cohom2}
\end{equation}
Now we can explain why we choose $F$ as the Hilbert transform of $H$ in~(\ref{eq: hilbert}). We select $F$ to inherit the most important information in $\hat{F}$, namely, the imaginary unit $i$ and $\mathrm{sgn}(k)$.
Readers can check that we still get the heuristic argument above if we choose the $\hat{F}$ whose Fourier coefficients are~(\ref{eq: cohom2}) to do the averaging in~(\ref{eq: cont}).
\subsection{The integral equation.}\label{subsection: exact solution}
Now we take into account the third term in the RHS of equation~(\ref{eq: fourier}). We first remove the $-|k|H^k$ term in equation~(\ref{eq: fourier}) by setting $H^k=e^{-|k|\dt}u^k$ to obtain
\begin{equation}
u^k_\dt=-i\eps\sm_k\{u^k,\langle H_1\rangle\}_{(x,y)}+2i\eps\sm_m\sum_{\substack{l+m=k,\\
m<0<l}}e^{-(|l|+|m|-|k|)\dt}\{u^l,u^m\}_{(x,y)}.\label{eq: v}
\end{equation}
If we define an operator $g^{is*}f:=f\circ g^{is}$, where $g^t$ is the flow generated by the Hamiltonian $-\langle H_1\rangle$, the exact solution of the truncated equation \[u^k_\dt=-i\eps\sm_k\{u^k,\langle H_1\rangle\}_{(x,y)}\] would be $g^{\eps\sm_ki\dt*}u^k(x,y,0)$. \\
Then using the variation of parameter method in ODE, we can write the exact solution to equation~(\ref{eq: v}) in the following form.
\begin{equation}
\begin{aligned}
&\displaystyle u^k(x,y,\dt)\\
&=g^{\eps\sm_ki\dt*}u^k(x,y,0)+2i\eps\sm_m\int_{0}^{\dt}e^{-(|l|+|m|-|k|)s}g^{\eps\sm_ki(\dt-s)*}\sum_{\substack{l+m=k,\\ m<0<l}}\{u^l,u^m\}_{(x,y)}\ ds\\
&=g^{\eps\sm_ki\dt*}u^k(x,y,0)\\
&+2i\eps\sm_m\int_{0}^{\dt}\sum_{\substack{l+m=k,\\ m<0<l}}e^{-(|l|+|m|-|k|)s}\{g^{\eps\sm_ki(\dt-s)*}u^l,g^{\eps\sm_ki(\dt-s)*}u^m\}_{(x,y)}\ ds.\\
\end{aligned}\label{eq: int}
\end{equation}
We will analyze this equation to study its solution. To do so, we need a good control of the non-homogeneous term, i.e. the second term in the RHS.
\subsection{Control of the nonhomogeneous term of equation~(\ref{eq: int}).}
To control the nonhomogeneous term, we use the majorant estimate. The majorant relation $``\ll"$ is defined as follows.
\begin{Def}
For any two functions $f(z)$, $g(z)$, $z=(z_1,z_2,\cdots,z_m)$, analytic at the point $z=0$,
\[f(z)=\sum_\bt f_\bt z^\bt,\quad g(z)=\sum_\bt g_\bt z^\bt,\]
\[\bt=(\bt_1,\bt_2,\cdots,\bt_m),\quad \bt_j\geq 0,\quad z^\bt=z_1^{\bt_1}\cdots,z_m^{\bt_m}.\]
We say that $g$ is a majorant for $f$ $(f\ll g)$ if for any multi-index $\bt$, we have $g_\bt\geq|f_\bt|$.\label{Def: maj1}
\end{Def}
The proof is first to guess a majorant assumption, then show the function in the assumption satisfies an equation that majorates the integral equation~(\ref{eq: int}). This checks the assumption and closes up the argument. \\
Now we make a majorant assumption
\begin{equation}
g^{\eps\sm_ki \tau *}u^k(x,y,\dt)\ll \mu e^{-|k|\rho}V(Y,\dt),\quad Y=x+y,\ |\tau|<\dt^*,
\label{eq: assumption}
\end{equation}
where $\dt^*\sim \rho/\eps$ is the maximal extension time determined by the homogeneous part of equation~(\ref{eq: fourier}) in the heuristic argument. The $e^{-|k|\rho}$ characterizes the way how the Fourier coefficients decay in the case of analytic perturbation and $\mu=\Vert H_1\Vert_\rho$.
We choose $Y=x+y$ to make it easier to calculate the derivatives since $\dfrac{\pt V}{\pt x}=\dfrac{\pt V}{\pt y}=\dfrac{\pt V}{\pt Y}$. \\
Then the integrand of equation~(\ref{eq: int}) can be majorated by $C(s)(V_Y)^2$, where \[\displaystyle C(s)=4\eps \mu\sum_{\substack{l+m=k,\\ m<0<l}}e^{-(|l|+|m|-|k|)(s+\rho)}\leq 4\mu\eps(1+1/\rho):=C.\]
This $C$ depends on the smoothness and magnitude of $H_1$ and the number of combinations $l+m=k$. The number of combinations of integers in one dimensional case is easy to estimate, but in higher dimensional case it becomes very difficult, which is the main difficulty that we need to overcome in this paper.

If we can solve the equation $V_\dt=CV_Y^2$, then equation~(\ref{eq: int}) can be viewed as\begin{equation}e^{|k|\rho}u^k\ll V(0)+\int_0^\dt V_s\ ds=V(0)+V(\dt)-V(0)=V(\dt).\label{eq: Schauder}\end{equation}
This checks the majorant assumption~(\ref{eq: assumption}).
In order to solve the equation $V_\dt=C(V_Y)^2$, we apply the operator $\dfrac{\pt}{\pt Y}$ to the equation. We get the following Burgers equation by setting $U=V_Y$.\\
\[U_\dt=2CUU_Y,\quad U(Y,0)=\dfrac{\sm }{\sm-Y}.\]
Here $U(Y,0)$ majorates $\nabla u$ in the sense $\nabla u^k(x,y,0)\ll Me^{-|k|\rho}U(Y,0)$. The initial condition $\dfrac{\sm }{\sm-Y}$ is due to Lemma~\ref{Lm: digression} in Appendix~\ref{section: majorant}, where $\sm$ is the width of analyticity in the slow variables $(x,y)$.
\subsection{Outcome of the continuous averaging procedure}
The Burgers equation can be solved explicitly using the characteristics method in PDE. The solution is
\begin{equation}U(Y,\dt)=\dfrac{2\sm }{(\sm-Y)+\sqrt{(\sm-Y)^2-8\sm C\dt}}.\label{eq: burgers}\end{equation}
In order to ensure $(\sm-Y)^2-8\sm C\dt\geq 0$, we obtain the maximal flow time given by the slow variables is $\dt<\dfrac{\sm}{8C}$. In fact $C=O(\eps)$, so combined with $\dt\leq \dt^*\sim\rho/\eps$ we get the maximal flow time is $O(1/\eps)$. We also notice $U$ is always bounded provided $|Y|$ is sufficiently small, so is $V$. Recall that we defined \[H^k=e^{-|k|\dt}u^k\ll e^{-|k|\dt-|k|\rho}V.\] Therefore each Fourier coefficient $H^k$ after the continuous averaging would be less than $e^{-\dt}=O(e^{-c/\eps})$ for some constant $c$. Adding up all these Fourier terms, we recover the Hamiltonian after the averaging, which is of order $O(e^{-c/\eps})$. This is the result proved in \cite{Tr1, Tr2, TZ}. We will work out all the details in Section~\ref{subsection: closeup}.
\section{Continuous averaging proof of the Normal form Theorem~\ref{Thm: main}}\label{section: main}
Now we prove Theorem~\ref{Thm: main} using the continuous averaging method. Let us go back to the setup in Section~\ref{section: normal}. Since we are looking at a motion that is very close to periodic orbit in the region of the phase space, the continuous averaging explained in the previous section could be applied. The periodic orbit corresponds to the fast angle $\tt$ in equation~(\ref{eq: egham}). The nonresonant part $\tilde{H}$ corresponds to the $\tt$ dependence term $\sum_{k\neq 0} H^k e^{ik\tt}$ in equation~(\ref{eq: egham}). The $\langle\om^*,I\rangle$ will produce the exponential decay in the same way as the term $I$ in equation~(\ref{eq: egham}) did in equation~(\ref{eq: exact}). And the term  $G(I)$ will generate the imaginary flow in the same way as the term $\langle H_1\rangle$ in equation~(\ref{eq: egham}) did in equation~(\ref{eq: exact}). Finally, the term $\bar{H}$ leads to additional difficulties.

We devote the remaining part of this section to the proof of Theorem~\ref{Thm: main}. The proof is organized as follows. 
\begin{itemize}
\item Set up the continuous averaging in terms of Hamiltonian and get some heuristic understanding of the averaging process in Section~\ref{subsection: Ham}. 
\item Apply it the Hamiltonian vector field in Section~\ref{subsection: vec}. 
\item Following procedures in Section~\ref{section: ca}, we define the operator $\mathbf{g}$ to write the differential equations as integral equations, then we write down the majorant equation and prove the majorant relations.
\item Derive necessary estimates in the theorem from the majorant estimates in Section~\ref{subsection: closeup}.
\end{itemize}
\subsection{Continuous averaging for Hamiltonian~(\ref{eq: hamiltonian})}\label{subsection: Ham}
\begin{figure}[hat]
\begin{center}
\includegraphics[width=0.6\textwidth]{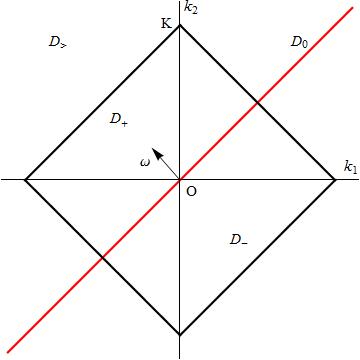}
\caption{Definition of sets when $\dt=0$. The diamond encloses integer vectors $|k|\leq K$. $\om$ is the frequency vector $\om^*$. The red line is a hyperplane perpendicular to $\om$. The region outside the diamond is $D_>(\dt)$. The red line splits the region inside the diamond into $D_+,D_-$.}
\end{center}\label{figure: def}
\end{figure}
In this section, we write down the continuous averaging and get a heuristic understanding. We start with a definition. As we have seen in Section~\ref{section: ca}, in the process of continuous averaging, we have different aspects like exponential decay, imaginary flow and nonhomogeneous terms.
\begin{Def} 
We define a partition of the width of analyticity $\rho$,\[ \rho=\rho_1+2\rho_2+\rho_3,\quad \rho_1,\rho_2,\rho_3>0,\]
and put \[ K=\dfrac{\rho_1}{\rho_3M_+\mathcal{R}\bar{T}}.\]
For $\dt>0,$ we also define the following sets to form a partition of the grid $\Z^n$.
\begin{equation}
\begin{aligned}\nonumber
& D_-(\dt)=\{l_-\in\Z^n|\quad\langle l_-,\om^*\rangle<0,\quad  |l_-|\rho_3+|\langle l_-,\om^*\rangle |\dt\leq \rho_3 K \},\\ \nonumber
& D_+(\dt)=\{l_+\in\Z^n|\quad \langle l_+,\om^*\rangle>0,\quad |l_+|\rho_3+|\langle l_+,\om^*\rangle |\dt\leq \rho_3 K \},\\ \nonumber
& D_0=\{l_0\in\Z^n|\quad \langle l_0,\om^*\rangle =0 \},\\ \nonumber
& D_>(\dt)=\Z^n\setminus (D_-(\dt)\cup D_+(\dt)\cup D_0).\\ \nonumber
\nonumber
\end{aligned}
\end{equation}
Finally, we define two functions of $\dt$ associated to the above sets.
\begin{equation}
\begin{aligned}
& \sm_k(\dt)=
\begin{cases}
\ 1 & k\in D_+(\dt),\\
\ -1 & k\in D_-(\dt),\\
\ 0 & k\in D_0\cup D_>(\dt).\\
\end{cases}\\
& S_k(\dt)=\int_0^{\dt}\sm_k(s)\ ds.\nonumber
\end{aligned}
\end{equation}\label{Def: main def}
\end{Def}
\begin{Rk}
\begin{itemize}
\item We split the analyticity width $\rho$ of the fast angle $\tt$ into $\rho=\rho_1+2\rho_2+\rho_3$. This splitting is quite flexible. We will optimize it to make $\rho_1$ as large as possible in Section~\ref{section: local} and~\ref{section: global}. Here $\rho_1$ would be used to control the imaginary flow, $\rho_3$ is used to do averaging, and $\rho_2$ is the remaining width of analyticity in angular variables after averaging. These distinctions will be made clear in the course of the proof.
\item We choose the cut-off $K$ to make sure that if $|k|\geq K$, then the corresponding Fourier coefficient is smaller than $e^{-\rho_3 K}$, which we think to be sufficiently small. A Fourier coefficient with $k\in D_\pm(\dt)$ will become smaller as $\dt$ increases. Once it is smaller than $e^{-\rho_3 K}$, the vector $k$ enters $D_>(\dt)$. So $D_\pm(\dt)$ keeps shrinking as $\dt$ increases. We stop running the continuous averaging once $D_\pm=\emptyset$.\\
\end{itemize}
\label{remark: main rk}
\end{Rk}
 We also define \[\dt^*:=\sup \{\dt|\ D_{\pm}(\dt)\neq \emptyset\}\] as the stopping time.
\begin{Lm}
The stopping time $\dt^*$ satisfies
\[\dt^*\leq K\bar{T}\rho_3.\]\label{Lm: dt*}
\end{Lm}
\begin{proof}
From the definition of $D_\pm(\dt)$, we know that at time $\dt^*$, we should have \[|l_+|\rho_3+|\langle l_+,\om^*\rangle |\dt^*= \rho_3 K.\]
We know $|\langle l_+,\om^*\rangle |\geq 1/\bar{T}$ from equation~(\ref{eq: tbar}). This implies $\dt^*\leq K\bar{T}\rho_3$.
\end{proof}
Now let us build our continuous averaging. This part is analogous to Section~\ref{subsection: derivation}. From equation~(\ref{eq: cont}), we have
\begin{equation}
\begin{aligned}
& H_\dt=\left\{F,H\right\}=\left\{F,\langle\om^*,I\rangle\right\}+\left\{F,G\right\}+\left\{F,\eps\bar{H}\right\}+\left\{F,\eps\tilde{H}\right\}.
\end{aligned}
\label{eq: main cont}
\end{equation}
\begin{Lm}
If we define \begin{equation}\displaystyle F:=\sum_{k\in\Z^n}i\eps\sm_k(\dt) H^k(I,x,y,\dt)e^{i\langle k,\tt\rangle}\label{eq: F}\end{equation} in the continuous averaging equation~$(\ref{eq: main cont})$, where $\sigma_k(\delta)$ is defined in Definition~$\ref{Def: main def}$, then depending on the properties of the Fourier mode $k$, we have the following three groups of PDEs.
\begin{subequations}
\begin{align}
&\textrm{For}\ k\in D_0,\\ \nonumber
&\displaystyle H^k_\dt e^{i\langle k,\tt\rangle}=\left\{F,\eps\tilde{H}\right\}^k=\sum_{l_{\pm}+l=k}i\sm_{l_{\pm}}\left\{H^{l_{\pm}}e^{i\langle l_{\pm},\tt\rangle},\eps H^le^{i\langle l,\tt\rangle}\right\},\quad l\in \Z^n\setminus D_0.\label{eq: a}\\
&\textrm{For}\ k\in D_-(\dt)\cup D_+(\dt),\\ \nonumber
&\displaystyle H^k_\dt e^{i\langle k,\tt\rangle}=-|\langle\om^*,k\rangle |H^k e^{i\langle k,\tt\rangle}+i\sm_k\left\{H^ke^{i\langle k,\tt\rangle},G\right\}+\left\{F,\eps\bar{H}\right\}^k+\left\{F,\eps\tilde{H}\right\}^k\\ \nonumber
& =-|\langle \om^*,k\rangle|H^ke^{i\langle k,\tt\rangle}+i\sm_k\left\{H^ke^{i\langle k,\tt\rangle},G\right\}+\sum_{l_{\pm}+l=k}i\sm_{l_{\pm}}\left\{H^{l_{\pm}}e^{i\langle l_{\pm},\tt\rangle},\eps H^le^{i\langle l,\tt\rangle}\right\}.\label{eq: b}\\
&\textrm{For}\ k\in D_>(\dt),\\ \nonumber
&\displaystyle H^k_\dt e^{i\langle k,\tt\rangle}=\left\{F,\eps\bar{H}\right\}^k+\left\{F,\eps\tilde{H}\right\}^k=i\sum_{l_{\pm}+l=k}\eps\sm_{l_{\pm}}\left\{H^{l_{\pm}}e^{i\langle {l_{\pm}},\tt\rangle}, H^le^{i\langle l,\tt\rangle}\right\}.\label{eq: c}
\end{align}\label{eq: 3 cont}
\end{subequations}
We have $l\neq 0$ in the latter two cases.\label{Lm: abc}
\end{Lm}
\begin{proof}
From the definition of $F$ we know the Fourier harmonics of $F$ come only from $D_\pm(\dt)$. As a result for any $k\neq 0$,
we must write $k=l_\pm+l$ for $l_\pm\in D_\pm(\dt)$ and some $l$. The equation~(\ref{eq: 3 cont}b) is straightforward.

If $k\in D_0$, then $\langle k, \om^*\rangle=0$. We know $\langle l_\pm,\om^*\rangle\neq 0$, then $\langle l,\om^*\rangle\neq 0$.
So in equation~(\ref{eq: 3 cont}a), the $H_0,\ \bar{H}$ terms do not appear.
If $k\in D_>(\dt)$, no Fourier harmonics from $\{F,\langle I,\om^*\rangle\}+\{F,G\}$ in equation~(\ref{eq: main cont}) enter
equation~(\ref{eq: 3 cont}c).
\end{proof}
Following directly from Definition~\ref{def: convex}, we have the lemma.
\begin{Lm}
If we define $\mathcal{R}$ as the confinement radius of $I$, i.e. $|I|_2\leq \mathcal{R},\ I\in G_n+\sm\subset\C^n$,
then we have the following estimates
\[|G|\leq M_+\mathcal{R}^2/2,\quad |\nabla G|_2\leq M_+\mathcal{R}.\]\label{Lm: g}
\end{Lm}
\begin{proof}
We first notice $G(0)=\nabla G(0)=0$.
For $|G|$, we use the formula
\[G(I)= \left\langle I\int_0^1(1-t)\nabla^2 H_0(tI)\ dt, I\right\rangle\]
and Definition~\ref{def: convex} to get the estimate in the lemma.
For $|\nabla G(I)|_2$, we use
\[\nabla G(I)= I\int_0^1 \nabla^2 G(tI)\ dt=I\int_0^1 \nabla^2 H_0(tI)\ dt.\]
\end{proof}

The following lemma helps us to understand the heuristic ideas of the process of continuous averaging and Definition~\ref{Def: main def}.
\begin{Lm}
If we omit the $\sum_{k=l_\pm+l}$ terms in the RHS of equations~$(\ref{eq: 3 cont})$, then equation~(\ref{eq: 3 cont}) can be solved explicitly and the solution satisfies
\[|H^k(I,x,y,\dt)|\leq \mu e^{-\rho|k|},\quad \mathrm{for}\ k\in D_>(0)\cup D_0, \]
\[|H^k(I,x,y,\dt)|\leq \mu e^{-2\rho_2|k|}e^{-\rho_3|k|-\dt|\langle \om^*,k\rangle|},\quad \mathrm{for}\ k\in D_\pm(0). \]
Moreover, at the stopping time $\dt^*$ we have
\[|H^k(I,x,y,\dt^*)|\leq \mu e^{-2\rho_2|k|}e^{-\rho_3K},\quad \mathrm{for}\ k\in D_\pm(0), \]
where the domain of variables is $(I,x,y)\in (\mathcal{G}^n+\sm)\times (\mathcal{W}^{2m}+\sm)$.\label{Lm: homo}
\end{Lm}
\begin{proof}
It follows from Definition~\ref{def: norm} that $|H^k|\leq \mu e^{-\rho|k|}$ for \[(I,x,y)\in (\mathcal{G}^n+\sm)\times (\mathcal{W}^{2m}+\sm).\]
If we truncate equations~(\ref{eq: 3 cont}), then the first and the third become $H^k_\dt=0$ for $k\in D_>(0)\cup D_0$. So we have the corresponding estimate of $|H^k|$ stated in the lemma.
However, equation~(\ref{eq: b}) becomes
{\small\[H^k_\dt e^{i\langle k,\tt\rangle}=-|\langle \om^*,k\rangle|H^ke^{i\langle k,\tt\rangle}+i\sm_k\{H^ke^{i\langle k,\tt\rangle},G\}=(-|\langle \om^*,k\rangle|-\sm_k\langle k,\nabla G\rangle)H^ke^{i\langle k,\tt\rangle}.\]
}
This equation admits an explicit solution
\[H^k(I,x,y,\dt)=e^{-|\langle \om^*,k\rangle|\dt-\sm_k\langle k,\nabla G\rangle\dt}H^k(I,x,y,0).\]
So we have the estimate
\[|H^k|\leq \mu e^{-\rho|k|}e^{-|\langle \om^*,k\rangle|\dt-\sm_k\langle k,\nabla G\rangle\dt},\quad k\in D_\pm(\dt).\]
Using Lemma~\ref{Lm: g}, we get \[|\sm_k\langle k,\nabla G\rangle|\leq |\nabla G|_\infty\cdot |k|\leq |\nabla G|_2\cdot |k|\leq M_+\mathcal{R}|k|.\]
In the splitting $\rho=\rho_1+2\rho_2+\rho_3$, we use $\rho_1$ to bound the term $\langle k,\nabla G\rangle$. Namely, we need \[|\sm_k\langle k,\nabla G\rangle|\dt\leq \rho_1|k|.\]
It is enough to require that
\begin{equation}\dt M_+\mathcal{R}\leq \rho_1.\label{eq: rho_1}\end{equation}
This also gives an upper bound for $\dt$. We equate this upper bound with the one given in Lemma~\ref{Lm: dt*} to obtain the value of $K$ in Definition~\ref{Def: main def}.
Now we have
\[|H^k|\leq \mu e^{-2\rho_2|k|}e^{-\rho_3|k|-|\langle \om^*,k\rangle|\dt},\quad k\in D_\pm(\dt).\]
The definition of $D_\pm(\dt)$ implies that once this $H^k$ term is already $e^{-2\rho_2 |k|}e^{-\rho_3 K}$, the $k$
  will enter $D_>(\dt)$ and not belong to $D_\pm(\dt)$ any more.
\end{proof}
\subsection{Continuous averaging for a vector field}\label{subsection: vec}
In order for the majorant estimates to be applicable to understand equations~(\ref{eq: 3 cont}), we need to write the continuous averaging equations in terms of Hamiltonian vector field.
\begin{Def}
We introduce the following vector fields $h^*,\ h_0,\ \bar{h},\ \tilde{h}$ corresponding to different parts $\langle I,\om^*\rangle,\ G,\ \bar{H},\ \tilde{H}$ of the Hamiltonian~$(\ref{eq: hamiltonian})$ and $f$ corresponding to $F$.
\begin{equation}
\begin{aligned}
& h^*=(0,\om^*,0,0),\qquad h_0=(0,\nabla G,0,0),\\
& \bar{h}=J\nabla \bar{H}, \quad\tilde{h}=J\nabla \tilde{H},\quad f=J\nabla F.\\
\end{aligned}
\end{equation}
We also use $h^k$ to denote the $k$-th Fourier coefficients of $\bar{h}$ and $\tilde{h}$. Moreover, corresponding to $F$ in equation~$($\ref{eq: F}$)$, we define \[f=\sum_ki\eps
\sm_k(\dt) h^ke^{i\langle k,\tt\rangle}.\]
\end{Def}
With this definition, we can rewrite the continuous averaging equation~(\ref{eq: main cont}) as follows by replacing the Poisson bracket
 by Lie bracket and the upper case letters $H,\ F$ by the lower case letters $h,\ f$ respectively.
\[h_\dt=[f,\om^*+h_0(I)+\eps\bar{h}+\eps\tilde{h}]\]

\begin{Lm}\label{Lm: abcv}
If we set $v^k=h^ke^{S_k(\dt)\langle\om^*,k\rangle}$
 $($recall $S_k(\dt)$ was defined in Definition~$\ref{Def: main def})$,
 then equations~$(\ref{eq: 3 cont})$ can be rewritten in the following form in terms of Hamiltonian vector field.
\begin{subequations}\label{eq: abcv}
\begin{align}
&\textrm{For}\ k\in D_0,\\ \nonumber
&v^k_\dt=e^{-i\langle k,\tt\rangle}i\eps\sum_{l_{\pm}+l=k}\sm_{l_{\pm}}\left[v^{l_{\pm}}e^{i\langle l_{\pm},\tt\rangle}, v^le^{i\langle l,\tt\rangle}\right]e^{-\left(\langle\om^*,l_{\pm}\rangle S_{l_{\pm}}+\langle\om^*,l\rangle S_l\right)}.\\
&\textrm{For}\ k\in D_-(\dt)\cup D_+(\dt),\\ \nonumber
&v^k_\dt =ie^{-i\langle k,\tt\rangle}\big(\sm_k\left[v^ke^{i\langle k,\tt\rangle},v_0\right]\\ \nonumber
& +\sum_{l_{\pm}+l=k}\eps \sm_{l_{\pm}}\left[v^{l_{\pm}}e^{i\langle l_{\pm},\tt\rangle},v^le^{i\langle l,\tt\rangle}\right]e^{-(\langle\om^*,l_{\pm}\rangle S_{l_{\pm}}+\langle\om^*,l\rangle S_{l}-\langle\om^*,k\rangle S_{k})}\big).\\
&\textrm{For}\ k\in D_>(\dt),\\ \nonumber
&v^k_\dt =ie^{-i\langle k,\tt\rangle}\sum_{l_{\pm}+l=k}\eps \sm_{l_{\pm}}\left[v^{l_{\pm}}e^{i\langle l_{\pm},\tt\rangle},v^le^{i\langle l,\tt\rangle}\right]e^{-(\langle\om^*,l_{\pm}\rangle S_{l_{\pm}}+\langle\om^*,l\rangle S_{l}-\langle\om^*,k\rangle S_k)}.\\ \nonumber
\end{align}
\end{subequations}
\end{Lm}
\begin{proof}
In equations~(\ref{eq: 3 cont}), we replace the Poisson bracket
 by Lie bracket and the upper case letters $H,\ F$ by the lower case letters $h,\ f$ respectively. Then we remove
 the $-|\langle \om^*,k\rangle|h^k$ in the second case as we did in Section~\ref{subsection: exact solution}.
 We set $v^k=h^ke^{S_k(\dt)\langle\om^*,k\rangle}$. Then direct computation proves the lemma.
\end{proof}
\subsection{The operator $\mathbf{g}$ and the majorant commutator}\label{subsection: operator}
What we do next is to write the differential equations for $v^k$'s as integral equations. As we did in Section~\ref{subsection: exact solution},
we first need to define an operator $g$
which solves the homogeneous part of equations~(\ref{eq: abcv}), esp.~(\ref{eq: abcv}b).
\subsubsection{The operator $\mathbf{g}$}
\begin{Def}[Section 2 of \cite{PT}]
Let $g^t$ be the Hamiltonian flow of the Hamiltonian vector field $h_0(I)$ generated by the Hamiltonian $G(I)$.
We put $f_k=\hat{f}(I,x,y)e^{i\langle k,\tt\rangle}$ for an arbitrary analytic function $\hat{f}$ defined on $\mathcal{D}(\rho,\sigma)$ and then define: \[\mathbf{g}^{t}_{k} \hat{f}=e^{-i\langle k,\tt\rangle}g_*^{-it}(f_k\circ g^{it}),\quad t\in\R.\]\label{def: g}
\end{Def}
It is shown in Section 5 and 7 of \cite{PT} that $\mathbf{g}$ has the following two properties.
\begin{itemize}
\item $v^k(I,\dt)=\mathbf{g}^{S_k(\dt)}_k v^k(I,0)$ solves $v^k_\dt(I,\dt)=i\sm_k(\dt) e^{-i\langle k,\tt\rangle}[v^ke^{i\langle k,\tt\rangle},v_0]$. \\
\item\[\mathbf{g}^t_{k_1+k_2}\left(e^{-i\langle k_1+k_2,\tt\rangle}[\hat{f}_1e^{i\langle k_1,\tt\rangle},\hat{f}_2e^{i\langle k_2,\tt\rangle}]\right)\\
=e^{-i\langle k_1+k_2,\tt\rangle}[e^{i\langle k_1,\tt\rangle}\mathbf{g}^t_{k_1}\hat{f}_1, e^{i\langle k_2,\tt\rangle}\mathbf{g}^t_{k_2}\hat{f}_2].\]
\end{itemize}
With this operator $\mathbf{g}$, we can write differential equations~(\ref{eq: abcv}) as integral equations.
\begin{Lm}
If we denote the $\sum_{k=l_\pm+l}$ terms in equations~$(\ref{eq: abcv}a,~\ref{eq: abcv}b,~\ref{eq: abcv}c)$ by $\eta_a,\ \eta_b\ \eta_c$ respectively,
then we have the following three integral equations equivalent to equations~$(\ref{eq: abcv})$.
\begin{subequations}
\begin{align}
\mathrm{For}\ k\in D_0,\quad v^k(I,\dt)&=v^k(I,0)+\eps\int_{0}^{\dt}(e^{-i\langle k,\tt\rangle}\eta_a^k)\ ds. \\
\mathrm{For}\ k\in D_-\cup D_+,\
v^k(I,\dt)&=\mathbf{g}_k^{\sm_k\dt}v^k(I,0)+\eps\int_{0}^{\dt}\mathbf{g}^{\sm_k(\dt-s)}_k (e^{-i\langle k,\tt\rangle}\eta_b^k)\ ds.\\
\mathrm{For}\ k\in D_>(\dt),\quad v^k(I,\dt)&=v^k(I,0)+\eps\int_{0}^{\dt}(e^{-i\langle k,\tt\rangle}\eta_c^k )\ ds.
\end{align}\label{eq: intv}
\end{subequations}
\end{Lm}
\begin{proof}
The equations~(\ref{eq: intv}a) and~(\ref{eq: intv}c) are straightforward. The equation~(\ref{eq: intv}b) is an application
of the first property of the operator $\mathbf{g}$ above and the variation of parameter method in ODE.
\end{proof}
\subsubsection{The majorant commutator}
We need the following majorant commutator to perform estimates.
\begin{Def}[Section 7 of \cite{PT}]
For any two functions $\hat{F},\hat{G}: \C^{n+2m}\rightarrow \C$, and any two vectors $l,k\in\Z^n$, we define the majorant commutator:
\[[[\hat{F},\hat{G}]]^{l,k}=(|l|+|k|)\hat{F}\hat{G}+(n+2m)\dfrac{\partial}{\partial Y}(\hat{F}\hat{G}),\] where $Y=I+x+y$.\label{Def: commutator}
\end{Def}
For this commutator, we have the following lemmas.
\begin{Lm}[Proposition 7.1 of \cite{PT}]
Suppose that $\hat{f}_1,\hat{f}_2:\C^{n+2m}\rightarrow \C^{2(n+m)}$, $\hat{F}_1, \hat{F}_2: \C^{n+2m}\rightarrow \C$, and
 $\hat{f}_1\ll \hat{F}_1$, $\hat{f}_2\ll \hat{F}_2$. Here the majorant relation $\hat{f}_i\ll \hat{F}_i$ means that $\hat{F}_i$ majorates each component of the vector $\hat{f}_i,\ i=1,2$.
Then for any $k_1,k_2\in \Z^n$, \[e^{-i\langle k_1+k_2,\tt\rangle}[\hat{f}_1e^{i\langle k_1,\tt\rangle},\hat{f}_2e^{i\langle k_2,\tt\rangle}]\ll[[\hat{F}_1,\hat{F}_2]]^{k_1,k_2}.\]\label{Lm: majcomm1}
\end{Lm}
\begin{Lm}[Proposition 7.2 of \cite{PT}]
Suppose that $\mathbf{g}^\tau_{k_1}\hat{f}_1(y)\ll \hat{F}_1(Y),\ \mathbf{g}^\tau_{k_2}\hat{f}_2(y)\ll \hat{F}_2(Y)$. Then \[\mathbf{g}^\tau_{k_1+k_2}\left(e^{-i\langle k_1+k_2,\tt\rangle}[\hat{f}_1e^{i\langle k_1,\tt\rangle},\hat{f}_2e^{i\langle k_2,\tt\rangle}]\right)\ll [[\hat{F}_1,\hat{F}_2]]^{k_1,k_2}.\]\label{Lm: majcomm2}
\end{Lm}
\subsection{Majorant equation, the derivation and the solution}\label{subsection: maj}
\subsubsection{Majorant control on the initial value}
We first have majorant control on the initial value.
\begin{Lm}
For $|\dt|\leq \dt^*$, and $\mathcal{R}<\sm$, $k\in\Z^n$, we have
\[ v^k(I,x,y,0)\ll \dfrac{ \mu\sm e^{-|k|\rho}}{\sm-Y},\quad \mathbf{g}^\dt v^k(I,x,y,0)\ll \dfrac{ \mu\sm e^{-|k|(\rho_3+2\rho_2)}}{\sm-Y}.\]\label{Lm: initial}
\end{Lm}
\begin{proof}
We first consider $v^k(I,x,y,0)=h^k$. We know $|H^k|, |h^k|_\infty\leq \mu e^{-\rho|k|}$ for $(I,x,y)\in (\mathcal{G}^n+\sm)\times(\mathcal{W}^{2m}+\sm)$ from the definition of $\mu$ in Definition~\ref{def: norm}.
Then we use Lemma~\ref{Lm: digression} (4) in Appendix~\ref{section: majorant} to obtain the majorant control of $v^k(I,x,y,0)$.

Now we consider the effect of $\mathbf{g}$. The operator $\mathbf{g}$ is defined by the Hamiltonian flow generated by the Hamiltonian $iG(I)$ in Definition~\ref{def: g}. The variables $I,\ x,\ y$
are constants of motion of this Hamiltonian flow. So $\mathbf{g}$ only shrinks the width of analyticity in $\tt$ but has no influence on that of $I,\ x,\ y$. From the definition of $\mathbf{g}$, we see
\[|\mathbf{g}^\dt v^k(I,x,y,0)|_\infty\leq e^{|\langle k,\nabla G\rangle|\dt}\left|v^k(I,x,y,0)\right|_\infty.\]
We also have \[|\langle k,\nabla G\rangle|\dt\leq M_+\mathcal{R}\dt^*|k|\leq \rho_1|k|\]
according to inequality~(\ref{eq: rho_1}). This tells us
\[|\mathbf{g}^\dt v^k(I,x,y,0)|_\infty\leq \mu e^{-(2\rho_2+\rho_3)|k|},\quad \mathrm{for}\ (I,x,y)\in (\mathcal{G}^n+\sm)\times(\mathcal{W}^{2m}+\sm).\]
Now use Lemma~\ref{Lm: digression} (4) in Appendix~\ref{section: majorant} again to obtain the lemma.\\
  \end{proof}

\subsubsection{Majorant equations}
The following construction is given in \cite{PT}.
\begin{Def}
Consider a continuous function $a(\dt)$. We define the functions $W$ and $W^{|k|}$ as follows as solutions of PDEs.
\begin{equation}
\begin{aligned}
& W_\dt=a(\dt)WW_Y,\quad W\big|_{\dt=0}=\dfrac{1}{\sm-Y},\\
& W^{|k|}(\dt)=W(\dt),\quad \mathrm{for}\ |k|\leq K,\\
& W_\dt^{|k|}=a(\dt)\left(WW^{|k|}_Y+\dfrac{|k|}{K}WW^{|k|}\right),\quad W^{|k|}\big|_{\dt=0}=\dfrac{1}{\sm-Y},\ \mathrm{for}\ |k|>K.\\
\end{aligned}\label{eq: main maj}
\end{equation}\label{def: w}
\end{Def}
\begin{Lm}
The solutions $W$ and $W^{|k|}$ are given explicitly by
\begin{equation}
\begin{aligned}
& W=\dfrac{2}{(\sigma-Y)+\sqrt{(\sigma-Y)^2-4A(\dt)}},\\
& W^{|k|}=We^{ WB(\dt)\frac{|k|}{K}},\quad \textrm{for}\ |k|>K,\\
& A(\dt)=\int_{0}^{\dt}a(s)\ ds.\end{aligned}
\end{equation}
The solutions are defined up to time $\dt^*$ and for $Y$ satisfying the restrictions.
\begin{equation} A(\dt^*)\leq \left(\sm-(\sqrt{n}+2\sqrt{m})\mathcal{R}\right)^2/4,\quad \vert Y\vert<(\sqrt{n}+2\sqrt{m})\mathcal{R}.\label{eq: restriction}\end{equation}\label{Lm: restriction}
\end{Lm}
\begin{proof}
The fact that $W$ and $W^{|k|}$ are exact solutions can be checked directly. To obtain the restriction for $\dt^*$, we need to ensure $(\sigma-Y)^2-4A(\dt)\geq 0$ so that the square root makes sense.

We want that when $\dt=\dt^*$, we still have $|I|_2,\ |x|_2,\ |y|_2\leq \mathcal R$. We know \[|Y|\leq |I|+|x|+|y|\leq \sqrt{n}|I|_2+\sqrt{m}|x|_2+\sqrt{m}|y|_2\leq (\sqrt{n}+2\sqrt{m})\mathcal{R}.\]
\end{proof}
\begin{Rk}\label{Rk: pde}
Let us try to understand the PDEs~$(\ref{eq: main maj})$ heuristically. Consider \begin{equation}U_t=WU_x+VU,\quad U(x,0)=U_0(x)\label{eq: rk pde}.\end{equation} The way to solve it is the characteristic method. The characteristics is given by $\dfrac{dx}{dt}=-W$. Then we are able to write the PDE in the form: $dU/dt=VU$. Then $U=U_0e^{\int Vdt}$. So we see that, $W$ determines how fast we approach the intersection of characteristics, while $V$ determines how $U$ grows.
\end{Rk}
\subsubsection{Proof of the majorant relation: $W,W^{|k|}$ majorate the solutions of equation~(\ref{eq: intv}).}
The main result of this section is summarized in the following proposition, which implies the solutions of equations~(\ref{eq: main maj}) majorate that of equation~(\ref{eq: intv}).
\begin{Prop}
For any $\tau$ such that $|\tau|+\dt\leq \dt^*$, we have the following majorant control of the solution $v^k(I,x,y,\dt)$ of equation~$(\ref{eq: intv})$\\
\[\mathbf{g}^{\tau}_kv^k(I,x,y,\dt)\ll \sm\mu e^{-|k|(\rho_3+2\rho_2)} W^{|k|}(Y,\dt),\quad \mathrm{for}\ k\in\Z^n\setminus\{0\},\]
and\[v^k(I,x,y,\dt)\ll \sm\mu e^{-\rho|k|} W^{|k|}(Y,\dt),\quad \mathrm{for}\ k\in D_0\cup D_>(0).\]
under the restriction~\ref{eq: restriction} coming from Lemma~$\ref{Lm: restriction}$.
(The expression of $a(\dt)$ and $A(\dt)$ will be given explicitly in Lemma~\ref{Lm: estimate}.) Moreover $A(\dt^*)$ is given by \begin{equation}A(\dt^*)= 6e^\sm \sm\eps \mu \bar{T}\dfrac{(2K)^{n-1}K^2\rho_3}{n!}\left(1+\dfrac{2n}{\rho_3K}\left(1+\ln\dfrac{2K^2\rho_3}{n}\right)\right).\label{eq: adt}\end{equation}
\label{Lm: prop}
\end{Prop}
\begin{proof}
We first cite Proposition $A.1$ in \cite{PT}.
\begin{Lm}[\cite{PT}]
Consider the functions $W, W^{|k|}$ defined in Definition~\ref{def: w}, then the following statements are true:
\begin{enumerate}
\item $1/(\sm-Y)\ll W\ll W_Y,$
\item $W\ll W^{|k|},$
\item $W_YW^{|k|}\ll WW^{|k|}_Y,$
\item $W^{|k|}\ll We^{\frac{\sm|k|}{K}},$
\item $W^{|k'|}\ll W^{|k|}e^{\sm(|k'|-|k|)/K},\quad |k|<|k'|.$
\end{enumerate}\label{Lm: tr trick}
\end{Lm}
Let us first divide equations~(\ref{eq: intv}a),~(\ref{eq: intv}b),~(\ref{eq: intv}c) by the numerators of the initial condition in Lemma~\ref{Lm: initial}, i.e. $\sm\mu e^{-\rho|k|}$, $\sm\mu e^{-|k|(\rho_3+2\rho_2)}$ and $\sm\mu e^{-\rho|k|}$ respectively.
Then we use the expression $\sm\mu\zeta^k$ to refer to any one
  of $\eps e^{\rho|k|}e^{-i\langle k,\tt\rangle}\eta_a^k$,
  $\eps e^{|k|(\rho_3+2\rho_2)}\mathbf{g}^{\dt-s}\left(e^{-i\langle k,\tt\rangle}\eta^k_b\right)$
  or $\eps e^{\rho|k|}e^{-i\langle k,\tt\rangle}\eta_c^k$, (see the integrands of equations~(\ref{eq: intv})
  for the definitions of $\eta_a,\ \eta_b,\ \eta_c$).\\
To carry out the proof, we substitute the majorant relation in Proposition~\ref{Lm: prop} into equations~(\ref{eq: intv}) to check that equations~(\ref{eq: intv}) are majorated by equations~(\ref{eq: main maj}). This is the plan proposed in \cite{PT}. We use the majorant commutator to majorate each of the Lie bracket of $\zeta^k$ according to Lemma~\ref{Lm: majcomm1},\ \ref{Lm: majcomm2},\\
\begin{equation}
\zeta^k\ll \eps\sm\mu\sum_{l_{\pm}+l=k}e^{-(|l_{\pm}|+|l|-|k|)(\rho_3+2\rho_2)}e^{-\left (S_{l_{\pm}}\langle\om^*,l_{\pm}\rangle +S_{l}\langle\om^*,l\rangle -S_{k}\langle\om^*,k\rangle\right)} [[W,W^{|l|}]]^{l_{\pm},l}. \nonumber
\end{equation}
Here we also use that $e^{-\rho(|l_{\pm}|+|l|-|k|)}\leq e^{-(|l_{\pm}|+|l|-|k|)(\rho_3+2\rho_2)}$ for $\eta_a$ and $\eta_c$. For simplicity, we denote the exponential weight by \[E(k,l_\pm,l,\dt):=e^{-(|l_{\pm}|+|l|-|k|)(\rho_3+2\rho_2)}e^{- \left(S_{l_{\pm}}\langle\om^*,l_{\pm}\rangle +S_{l}\langle\om^*,l\rangle -S_{k}\langle\om^*,k\rangle\right)} .\]
Applying the definition of the majorant commutator (Definition~\ref{Def: commutator}), we get \\
\begin{equation}
\displaystyle
\zeta^k\ll \eps\sm\mu\sum_{l_{\pm}+l=k}E(k,l_\pm,l,\dt)\left((|l_{\pm}|+|l|)WW^{|l|}+2(n+2m)WW^{|l|}_Y\right).\nonumber
\end{equation}
Here we use Lemma~\ref{Lm: tr trick}(3). This gives $(WW^{|l|})_Y\ll 2WW^{|l|}_Y$.
We introduce the notations \\
\begin{equation}
\begin{aligned}
&\Sigma_{\pm}:=\displaystyle\sum_{l_++l_-=k}E(k,l_{+},l_-,\dt),\\
&\Sigma_>:=\displaystyle\sum_{l_{\pm}+l_>=k}E(k,l_{\pm},l_>,\dt),\\
&\Sm_0:=\displaystyle\sum_{l_{\pm}+l_0=k}E(k,l_{\pm},l_0,\dt),
\end{aligned}\label{eq: sum}
\end{equation}
to obtain\begin{equation}
\begin{aligned}
&\zeta^k\ll 2\eps\sm\mu\Sm_\pm\left((|l_+|+|l_-|)W^2+(n+2m)(W^2)_Y\right)\\
& +\eps\sm\mu\Sm_>\left((|l_{\pm}|+|l_>|)WW^{|l_>|}+2(n+2m)WW^{|l_>|}_Y\right)\\
& +\eps\sm\mu\Sm_0\left((|l_{\pm}|+|l_0|)WW^{|l_0|}+2(n+2m)WW^{|l_0|}_Y\right).\\
\end{aligned}
\nonumber
\end{equation} \\
The second term in the RHS is the most complicated one. We only consider this term. The other two terms are done similarly.
\begin{equation}
\begin{aligned}
\displaystyle
&\eps\sm\mu\Sm_>\left((|l_{\pm}|+|l_>|)WW^{|l_>|}+2(n+2m)WW^{|l_>|}_Y\right)\\
&\ll \eps\sm\mu\Sm_>\left((2K+|k|)WW^{|k|+K}+2(n+2m)WW^{|k|+K}_Y\right)\\
&\ll \eps\sm\mu\Sm_>\left(|k|WW^{|k|+K}+2(K+n+2m)WW^{|k|+K}_Y\right)\\
&\ll 3Ke^\sm\eps\sm\mu\Sm_>\left(\frac{|k|}{K}WW^{|k|}+WW^{|k|}_Y\right).\\
\end{aligned}\label{eq: Kn}
\end{equation}
Here $|l_>|\leq K+|k|$, because $l_>=k-l_{\pm},\ |l_{\pm}|\leq K$. We used Lemma~\ref{Lm: tr trick}(5) to decrease the exponent of $W^{|l|}$. We also imposed a mild restriction:\begin{equation} 2(n+2m)\leq K.\qquad \label{eq: *}\end{equation}
If $|k|\geq K$, we get the majorant equation for the $W^{|k|}$ part in equation~(\ref{eq: main maj}).\\
If $|k|\leq K$, using Lemma~\ref{Lm: tr trick}(1) and $W^{|k|}=W$, we replace the last ``$\ll$" in ~(\ref{eq: Kn}) by
\[\ll 6Ke^\sm\eps\sm\mu\Sm_>WW_Y.\]
This is the majorant equation for $W$ in equation~(\ref{eq: main maj}).

For $\Sm_\pm$ and $\Sm_0$, we get the same majorant estimate with $\Sm_>$ replaced by $\Sm_\pm$ and $\Sm_0$.

Now the problem is to find $a(\dt)$ to give bound for  $6Ke^\sm\eps\sm\mu(2\Sm_\pm+\Sm_0+\Sm_>)$. We need to do some careful analysis for this and the result is summarized in the following lemma.\\
\begin{Lm}
We have the following upper bound for $(2\Sm_\pm+\Sm_0+\Sm_>)$, where $\Sm_\pm,\ \Sm_0,\ \Sm_>$ are defined in~$(\ref{eq: sum})$,
$$
(2\Sm_\pm+\Sm_0+\Sm_>)(\dt)\leq \dfrac{(2K-2\dt/\bar{T}\rho_2)^{n-1}}{(n-1)!}+
\begin{cases}
\dfrac{2(2K)^{n}}{n!} & \textrm{if}\ \dt\leq \dfrac{\bar{T}n}{2K},\\
\dfrac{2(2K)^{n-1}\bar{T}}{(n-1)!\dt} & \textrm{if}\ \dt\geq \dfrac{\bar{T}n}{2K}.
\end{cases}
$$
If we define $a(\dt)=6Ke^\sm\eps\sm\mu(2\Sm_\pm+\Sm_0+\Sm_>)$, then \[A(\dt^*)\leq A(K\bar{T}\rho_3)=\displaystyle\int_{0}^{K\bar{T}\rho_3} a(s)ds\] is equation~$(\ref{eq: adt})$ in Proposition~$\ref{Lm: prop}$.
\label{Lm: estimate}
\end{Lm}
The proof of this lemma is given in Appendix~\ref{section: appendix}.

This lemma gives the restriction~$(\ref{eq: adt})$ in Proposition~\ref{Lm: prop}. What we have shown is that each integrand of equations~(\ref{eq: intv}) has majorant estimate \[e^{-|k|(\rho_3+2\rho_2)}\mu\zeta^k\leq W^{|k|}_\dt,\] where $W^{|k|}$ satisfies equations~(\ref{eq: main maj}). Combined with the majorant control on initial condition in Lemma~\ref{Lm: initial}, this implies the LHS of equations~(\ref{eq: intv}) is majorated by $W^{|k|}$ . Now the proof of the proposition is complete.\\
\end{proof}
\subsection{The system after the averaging}\label{subsection: closeup}
The continuous averaging gives us the following information about the Hamiltonian vector fields.
\begin{Lm}
At time $\dt=\dt^*$, we have \[h^k\ll \sm\mu e^{-|k|\rho}W^{|k|},\qquad \mathrm{for}\ k\in D_0\cup D_>(0),\]
\[h^k\ll \sm\mu e^{-2\rho_2|k|}e^{-\rho_3K}W^{|k|},\quad \mathrm{for}\ k\in D_\pm(0).\]\label{Lm: after}
\end{Lm}
\begin{proof}
Recall in Lemma~\ref{Lm: abcv}, we set $v^k=h^ke^{S_k(\dt)\langle\om^*,k\rangle}$. Using the definition of $S_k(\dt)$ in Definition~\ref{Def: main def}, we get $v_k=h_k$ for $k\in D_0\cup D_>(0)$. Then Proposition~\ref{Lm: prop} applies to such $k$'s. For $k\in D_\pm(0)$, we must have $\rho_3|k|+S_k(\dt^*)\langle \om^*,k\rangle=\rho_3 K$ according to the definition $D_\pm(\dt)$. Then apply Proposition~\ref{Lm: prop} to this case. \\
\end{proof}
\subsubsection{The estimate of the normal form}
Now we use the information that we have obtained to prove Theorem~\ref{Thm: main}. Let us define the change of variables \[(I,\tt,x,y)(0)\to (I,\tt,x,y)(\dt^*):=(I',\tt',x',y')\] obtained by the continuous averaging at the stopping time $\delta=\dt^*$. Then, the Hamiltonian~(\ref{eq: hamiltonian}) in these new variables is of the form
\[H'(I',\tt',x',y')=H_0(I')+\eps\bar{\Psi}(I',\tt',x',y')+\eps\tilde{\Psi}(I',\tt',x',y'),\]
where $\bar{\Psi}$ is the resonant term and $\tilde{\Psi}$ is the nonresonant term as defined in Definition~\ref{def: resonant}. The following lemma gives estimates for the functions $\bar{\Psi}$ and $\tilde{\Psi}$.
\begin{Lm}
Suppose $K\geq 2\sm/(5\rho_1)$, $K\geq 2(n+2m)$, $5(\sqrt{n}+2\sqrt{m})\mathcal{R}\leq \sm$ and $A(\dt^*)<4\sm^2/25$. Denote after the averaging, $\bar{H}\to\bar{H}(\dt^*):=\bar{\Psi}$ and $\tilde{H}\to\tilde{H}(\dt^*):=\tilde{\Psi}$, then
\[\left\Vert\tilde{\Psi}\right\Vert_{\rho_2},\ \left\Vert\dfrac{\partial\tilde{\Psi}}{\partial \tt}\right\Vert_{\rho_2}\leq \dfrac{5\mu}{\rho_2^n}e^{-\frac{\rho_1}{M_+\mathcal{R}\bar{T}}},\quad \left\Vert\bar{\Psi}\right\Vert_{\rho_2}\leq \dfrac{5\mu}{\rho_2^n},\quad \left\Vert \bar{\Psi}+\tilde{\Psi}\right\Vert_{\rho_2}\leq \dfrac{5\mu}{\rho_2^n},\]
where $(I,x,y)\in (\mathcal{G}^n+4\sm/5)\times (\mathcal{W}^{2m}+4\sm/5)$ after the averaging.\label{Lm: Phi}
\end{Lm}
\begin{proof}
Notice the $I$ component of $h^k$ is $ikH^k$. So Lemma~\ref{Lm: after} implies
\[\left|k\Psi^k\right|_\infty,\ \left|\Psi^k\right|\leq \sm\mu e^{-\rho|k|}W^{|k|},\ \mathrm{for} \ k\in D_0\cup D_>(0),\]
\[\left|\Psi^k\right|\leq \sm\mu e^{-2\rho_2|k|}e^{-\rho_3 K}W,\quad \mathrm{for}\ k\in D_\pm(0),\]
where $\Psi^k$'s are Fourier coefficients of $\bar{\Psi}$ and $\tilde{\Psi}$. First we get upper bound for $W$. From Lemma~\ref{Lm: restriction}, for $|Y|\leq (\sqrt{n}+2\sqrt{m})\mathcal{R}$ we have
\[W\leq \dfrac{2}{\sm-(\sqrt{n}+2\sqrt{m})\mathcal{R}}\leq 5/(2\sm)\]
provided $5(\sqrt{n}+2\sqrt{m})\mathcal{R}\leq \sm$. (We will see in Section~\ref{section: global} that the confinement radius $\mathcal{R}=o(1)$ as $\eps\to 0$, so this condition is easy to satisfy.) The remaining width of analyticity for $(I',x',y')$ becomes $4\sm/5$. We can also replace the condition $A(\dt^*)<(\sm-(\sqrt{n}+2\sqrt{m})\mathcal{R})^2/4$ by
\[A(\dt^*)<\dfrac{4\sm^2}{25}.\]

The factor $e^{-\rho_2|k|}$ is absorbed in the $\Vert\cdot\Vert_{\rho_2}$ norm. For $k\in D_\pm(0)$, we are left with $\mu e^{-(\rho_3+\rho_2) K}W$. For $k\in D_>(0)$, we have $W^{|k|}=We^{WA(\dt^*)\frac{|k|}{K}}$ according to~(\ref{eq: main maj}). We only need to ensure $WA(\dt^*)/K\leq\rho_1$
so that $e^{WA(\dt^*)|k|/K}\leq e^{\rho_1|k|}$. This is $\sm/K\leq 5\rho_1/2$ because $A(\dt^*)\leq 4\sm^2/25$ and $W\leq 5/(2\sm)$. (We will also see in Section~\ref{section: global} that $K\to\infty$ as $\eps\to 0$, so this condition is also easy to satisfy.)\\
\[\left\Vert\tilde{\Psi}\right\Vert_{\rho_2}\leq \sm\mu W\left(e^{-\rho_3 K}\sum_{|k|\leq K}e^{-\rho_2|k|}+\sum_{|k|>K}e^{-(\rho_3+\rho_2)|k|}\right)\]\[\leq \sm\mu We^{-\rho_3 K}\left(\dfrac{1}{\rho_2^n}+\dfrac{e^{-\rho_2 K}}{(\rho_3+\rho_2)^n}\right)\leq 2\sm\mu We^{-\rho_3 K}/\rho_2^n\leq \dfrac{5\mu}{\rho_2^n}e^{-\rho_3K}.\]
Recalling the definition of $K=\dfrac{\rho_1}{\rho_3M_+\mathcal{R}\bar{T}}$ in Definition~\ref{Def: main def},
we find \[\left\Vert\tilde{\Psi}\right\Vert_{\rho_2}\leq \dfrac{5\mu}{\rho_2^n}\exp\left(-\dfrac{\rho_1}{M_+\mathcal{R}\bar{T}}\right).\]

Similarly, we have the estimates for $\bar{\Psi},\ \bar{\Psi}+\tilde{\Psi},\ \dfrac{\partial \tilde{\Psi}}{\partial\tt}$. \end{proof}
\subsubsection{The deviation of action variables in the real domain}
\begin{Lm}
Under the same hypothesis as Lemma~\ref{Lm: Phi}, after the averaging the total deviation of the variables is
\[\big| (I',\tt',x',y')-(I,\tt,x,y)|_\infty\leq \dfrac{5\eps\mu T}{2\pi(\rho_3+\rho_2)^n}.\]\label{Lm: dI}
Here the norm $|\cdot|_\infty$ is taken in the real domain.
\end{Lm}
\begin{proof}
For simplicity, we consider only the $I$ component. The other components are similar. From equation~(\ref{eq: JdF}), we have \[\dfrac{dI}{d\dt}=\{I,F\}=\dfrac{\partial F}{\pt\tt},\]
where the RHS is a real function. Then
\begin{equation}I'-I= \int_0^{\dt^*}\dfrac{\partial F}{\pt\tt}\ d\dt.\label{eq: dt I}\end{equation}
We have
\[\dfrac{\partial F}{\pt\tt}=-\eps\sum_{|k|\leq K}\sm_k(\dt)k H^ke^{i\langle k,\tt\rangle},\] since $F=\eps\sum_{|k|\leq K}i\sm_k(\dt)H^ke^{i\langle k,\tt\rangle}$ defined in Lemma~\ref{Lm: abc}. We also have
\[\sm_k(\dt)kH^k\ll \sm\mu e^{-(2\rho_2 +\rho_3)|k|}e^{-\dt|\langle k,\om^*\rangle|}W(Y,\dt).\]
Hence we have the estimate (recall $2\pi \bar T=T$.)
\[\left\Vert \dfrac{\partial F}{\pt\tt}(\dt)\right\Vert_{\rho_2}\leq\eps\sm\mu\sum_{|k|\leq K} e^{-(\rho_3+\rho_2)|k|}e^{-\dt|\langle k,\om^*\rangle|}W\leq \dfrac{5\eps\mu}{(\rho_3+\rho_2)^n}e^{-\dt/\bar{T}}.\]

\[\left\Vert  \int_0^{\dt^*}\dfrac{\partial F}{\pt\tt}(\dt)d\dt\right\Vert_{\rho_2}\leq \int_0^{\dt^*}\left\Vert  \dfrac{\partial F}{\pt\tt}(\dt)\right\Vert_{\rho_2}d\dt\leq \dfrac{5\eps\mu}{(\rho_3+\rho_2)^n}\int_0^{\dt^*}e^{-\dt/\bar{T}}d\dt\leq \dfrac{5\eps\mu\bar{T}}{(\rho_3+\rho_2)^n}.\]
 \end{proof}
\begin{proof}[Proof of Theorem~\ref{Thm: main}]
Lemma~\ref{Lm: Phi} and~\ref{Lm: dI} complete the proof of the theorem. Notice the conditions of the theorem coincide with that of Lemma~\ref{Lm: Phi} and~\ref{Lm: dI}, where the last condition in the theorem is exactly $A(\dt^*)<4\sigma^2/25$. Lemma~\ref{Lm: Phi} gives the estimate of $\tilde\Psi$ and Lemma~\ref{Lm: dI} gives the estimates for the deviation of the variables. 
\end{proof}

\section{Local stability: stability in a vicinity of a given periodic orbit }\label{section: local}
In this section, we derive stability result using the normal form Theorem~\ref{Thm: main}. Recall in Section~\ref{section: normal}, we have set $\om^*=\dfrac{\partial H_0}{\partial I}(0)$ as the frequency vector of the periodic orbit that we are considering. We consider initial condition $I(0)$ such that $|I(0)|_2\leq r$.
\begin{Thm}(Local stability)\label{Thm: local}
If the conditions of Theorem~\ref{Thm: main} and the following conditions are satisfied, 
\begin{itemize}
\item $\dfrac{6\eps\mu}{\rho_2^n} M_-\leq M_+^{2} r^2$,
\item $\dfrac{5n\eps\mu \bar{T}}{(\rho_2+\rho_3)^n}\leq r$,
\end{itemize}
then we have stability result: if the initial conditions $|I(0)|_2\leq r$, then one has $| I(t)|_2\leq \mathcal{R}:=8\dfrac{rM_+}{M_-}$, for all time $|t|\leq \mathcal{T}:=\dfrac{1}{|\om^*|}e^{\frac{\rho_1}{M_+\mathcal{R}\bar{\bar{T}}}}$.\\
\end{Thm}
\begin{proof}

The integrable part has the Taylor expansion around the point $I'(0)$,
  \begin{equation}
  \begin{aligned}
  &H_0(I'(t))-H_0(I'(0))=\langle\om(I'(0)),I'(t)-I'(0)\rangle\\
  &+\left\langle\int_0^1(1-s)\nabla^2 H_0\left(sI'(t)+(1-s)I'(0)\right)ds (I'(t)-I'(0)), I'(t)-I'(0)\right\rangle.
  \end{aligned}\nonumber
  \end{equation}
We obtain the following inequality using Definition~\ref{def: convex}
\begin{equation}\quad\qquad\dfrac{1}{2}M_-\left|I'(t)-I'(0)\right|_2^2\leq \left|H_0(I'(t))-H_0(I'(0))\right|+\left|\langle\om(I'(0)),I'(t)-I'(0)\rangle\right|.\label{eq: energy}\end{equation}
We use the energy conservation and Lemma~\ref{Lm: Phi} for the first term of the RHS to get
\[\left|H_0(I'(t))-H_0(I'(0))\right|\leq \eps\left(\left\Vert\left(\bar{\Psi}+\tilde{\Psi}\right)(I'(0))\right\Vert_{\rho_2}+\left\Vert\left(\bar{\Psi}+\tilde{\Psi}\right)(I'(t))\right\Vert_{\rho_2}\right)\leq \dfrac{10\eps\mu}{\rho_2^n}.\]
For the second term in the RHS of inequality~(\ref{eq: energy}), we have
\[\left|\langle\om(I'(0)),I'(t)-I'(0)\rangle\right|\leq \left|\langle\om^*,I'(t)-I'(0)\rangle\right|+\left|\langle\om(I'(0))-\om^*,I'(t)-I'(0)\rangle\right|,\]
  For the first term in the RHS, we use the Hamiltonian equation, Lemma~\ref{Lm: Phi} and the fact that $\left\langle\om^*,\dfrac{\partial\bar{\Psi}}{\partial\tt}\right\rangle=0$.
  \[\left|\langle\om^*,I'(t)-I'(0)\rangle\right|\leq |t|\left\Vert\left\langle \om^*,\dfrac{\pt\eps\tilde{\Psi}}{\pt \tt}\right\rangle\right\Vert_{\rho_3}\leq \mathcal{T}\dfrac{5\eps\mu}{\rho_2^n}e^{-\frac{\rho_1}{M_+\mathcal{R}\bar{T}}}|\om^*|,\quad \mathrm{for}\ |t|\leq \mathcal{T}.\]
  For the second term,
  \[\left|\langle\om(I'(0))-\om^*,I'(t)-I'(0)\rangle\right|\leq M_+\left(|I(0)|_2+|I'(0)-I(0)|_2\right)\cdot| I'(t)-I'(0)|_{2}\]\[\leq M_+\left(r+n\dfrac{5\eps\mu \bar{T}}{(\rho_3+\rho_2)^n}\right)\cdot| I'(t)-I'(0)|_{2},\]
  where we use $|I(0)|_2\leq r$ and $|I'(0)-I(0)|_2\leq n\dfrac{5\eps\mu \bar{T}}{(\rho_3+\rho_2)^n}$ following from Lemma~\ref{Lm: dI}. We have a factor $n$ since we go from $|\cdot|_\infty$ to $|\cdot|_2$.

If we set $a=|I'(t)-I'(0)|_2$, then we get an inequality of $a$ from~(\ref{eq: energy}):\[\dfrac{M_-}{2}a^2\leq \dfrac{10\eps\mu}{\rho_2^n}+\mathcal{T}\dfrac{5\eps\mu}{\rho_2^n}e^{-\frac{\rho_1}{M_+\mathcal{R}\bar{T}}}|\om^*|
 +M_+\left(r+n\dfrac{5\eps\mu \bar{T}}{(\rho_3+\rho_2)^n}\right)a.\]
We choose \begin{equation}
 \mathcal{T}\leq \dfrac{1}{|\om^*|}e^{\frac{\rho_1}{M_+\mathcal{R}\bar{T}}},\quad n\dfrac{5\eps\mu \bar{T}}{(\rho_3+\rho_2)^n}\leq r,\label{eq: tr}
\end{equation}
to obtain \[a\leq\dfrac{2M_+r+\sqrt{4M^{2}_+r^2+\frac{30\eps\sm\mu}{\sm\rho_2^n} M_-}}{M_-}.\]
We set \begin{equation}\frac{6\eps\mu}{\rho_2^n} M_-\leq M^{2}_+ r^2.\label{eq: r}\end{equation} Then we have
\[a=| I'(t)-I'(0)|_{2}\leq 5\dfrac{M_+ r}{M_-},\]
\[| I(t)-I(0)|_2\leq |I(t)-I'(t)|_2+| I'(t)-I'(0)|_2+|I'(0)-I(0)|_2\leq 2r+\dfrac{5rM_+}{M_-},\]
\[| I(t)|_2\leq |I(0)|_2+| I(t)-I(0)|_2\leq 3r+5\dfrac{rM_+}{M_-}\leq \dfrac{8rM_+}{M_-}.\]
The proof is now complete.
\end{proof}

\begin{Rk}
We introduce a restriction~\ref{eq: r} instead of introducing a constant $g$ as did in \cite{LNN,N}.
The two restrictions of this theorem implies $\rho_2,\rho_3$ can be sufficiently small if $\eps$ is. Then $\rho_1$ can be very close to $\rho$. The restrictions of Theorem~\ref{Thm: main} are also satisfied for $\eps$ small enough. Then we get improved stability time compared with \cite{LNN,N}.
 \label{Rk: rho}\\
\end{Rk}
\section{Global stability: stability for arbitrary initial data}\label{section: global}
In this section, we consider stability result for arbitrary initial data and give a proof of Theorem~\ref{Thm: global}. We first prove the following lemma.
\begin{Lm}\label{Lm: dirichlet}
Let us fix $(1<)Q\in \R$ and assume $\dfrac{\sqrt{n-1}}{Q^{1/(n-1)}}<\dfrac{M_-^2}{4\sup_{I\in\mathcal{G}^n}|\nabla^3 H_0|_\infty}$, then for any $I\in \mathcal{G}^n$, there exists an integer $q$, $1\leq q<  Q$, and a point $I^*\in \mathcal{G}^n$
such that $\left| I -I^* \right|_{2}\leq \dfrac{\sqrt{n-1}}{M_-\bar{T}Q^{1/(n-1)}}$ and $\om(I^*)$ is rational vector of period $\bar{T}=q/|\om(I)|_\infty$.
\end{Lm}
\begin{proof} First recall the Dirichlet theorem for simultaneous approximation:  \\
\emph{for any $\alpha\in \R^n, Q\in \R,\ and \ Q>1$, there exists an integer $q$,
$1\leq q<Q$, such that $| q\alpha-\Z^n|_{\infty}\leq Q^{-1/n}$.}

An improvement of the estimate can be obtained by rescaling $\al$ to $\al/|\al|_\infty$. Then apply the Dirichlet theorem to approximate the remaining $n-1$ components of $\al$ with one of whose $\pm 1$ entries removed.
We get the following:

\emph{ there exists a rational vector $\al^*$ of period $\bar{T}=\frac{q}{|\alpha|_\infty}$, $q\in \N$,
 $1<q<Q$ and $|\al^*- \alpha|_{\infty}\leq \dfrac{1}{\bar{T}Q^{1/(n-1)}}$ $($see (the only) Proposition in \cite{N}$)$.
}

The frequency vector is $\om(I)=\nabla H_0(I)$. Consider two points $I^*$ and $I$ such that $\om(I^*)$ is as stated in the lemma and approximates $\om(I)$ in the same way as $\al^*$ approximates $\al$.
\[\om(I)-\om(I^*)=\int_0^1\nabla \om(tI+(1-t)I^*)dt(I-I^*).\]
Hence from Definition~\ref{def: convex}\[M_-\vert I-I^*\vert_2^2\leq \left\langle\int_0^1\nabla \om(tI+(1-t)I^*)dt(I-I^*),(I-I^*)\right\rangle\]\[=\langle(I-I^*), \om(I)-\om(I^*) \rangle\leq \vert I-I^*\vert_2\vert \om(I)-\om(I^*)\vert_2.\]
This implies
\[M_-\vert I-I^*\vert_2\leq\vert \om(I)-\om(I^*)\vert_2\leq \sqrt{n-1}|\om(I)-\om(I^*)|_\infty\leq \dfrac{\sqrt{n-1}|\om(I)|_\infty}{qQ^{1/(n-1)}}.\]
In order to make sure the point $I^*$ can be found given $I$, we need to show the frequency map $\om(I)$ can be inverted, which can be done within the ball $B\left(\om(I^*), \dfrac{M_-^2}{4|\nabla^3 H_0|_\infty}\right)$ centered at $\om(I^*)$ with radius $\dfrac{M_-^2}{4|\nabla^3 H_0|_\infty}$ using implicit function theorem (see \cite{LNN}). So we assume $\dfrac{\sqrt{n-1}}{Q^{1/(n-1)}}<\dfrac{M_-^2}{4|\nabla^3 H_0|_\infty}$.
\end{proof}
\begin{proof}[Proof of Theorem~\ref{Thm: global}]

We define
$r=\dfrac{\sqrt{n-1}}{M_-\bar{T}Q^{1/(n-1)}}$ and $\mathcal{R}=8r\dfrac{M_+}{M_-}=\dfrac{8\sqrt{(n-1) }M_{+}}{M_-^2\bar{T}Q^{\frac{1}{n-1}}}$ as in Lemma~\ref{Lm: dirichlet}.
If we set $Q=\eps^{-\frac{n-1}{2n}}$, then we have \[\mathcal{R}=\dfrac{8\sqrt{n-1 }M_{+}}{M_-^2\bar{T}}\eps^{1/2n}\leq \dfrac{8\sqrt{n-1 }M_{+}}{M_-^2}\eps^{1/2n}|\nabla H_0|_\infty,\quad \mathcal{R}\bar{T}=\dfrac{8\sqrt{n-1 }M_{+}}{M_-^2}\eps^{1/2n}.\]

The stability time in Theorem~\ref{Thm: local} is
\[\mathcal{T}=\dfrac{1}{|\om^*|}\exp\left(\dfrac{\rho_1}{M_+\mathcal{R}\bar{T}}\right)\geq \dfrac{1}{\sup_{I\in\mathcal{G}^n}|\nabla H_0|}\exp\left(\dfrac{\rho_1M_-^2}{ 8\sqrt{n-1}M_+^2} \eps^{-\frac{1}{2n}}\right).\]
Now let us analyze the restrictions that we have. The restrictions are from Theorem~\ref{Thm: main}, Theorem~\ref{Thm: local} and Lemma~\ref{Lm: dirichlet}. The quantities $\bar{T},r,\mathcal{R}, K$ satisfy the following:
\[1<\bar{T}|\om|_\infty<Q=\eps^{-\frac{n-1}{2n}},\quad K=\dfrac{\rho_1}{\rho_3 \mathcal{R}\bar{T}}=\dfrac{\rho_1 M_-^2}{\rho_3 8\sqrt{n-1 }M_{+}}\eps^{-1/2n}.\]
\[\eps^{1/2}|\om|_\infty\dfrac{\sqrt{n-1}}{M_-}\leq r\leq \eps^{1/2n}|\om|_\infty\dfrac{\sqrt{n-1}}{M_-}.\]
\[\eps^{1/2}|\om|_\infty\dfrac{8\sqrt{n-1}M_+}{M^2_-}\leq \mathcal{R}\leq \eps^{1/2n}|\om|_\infty\dfrac{8\sqrt{n-1}M^+}{M^2_-}.\]
We substitute the bounds for $\bar{T},r,\mathcal{R}, K$ into the restrictions that we have to obtain the restriction of Theorem~\ref{Thm: global}. The first restriction in Theorem~\ref{Thm: global} is from the first one of Theorem~\ref{Thm: local}. The second in Theorem~\ref{Thm: global} is a collection of the first two ones of Theorem~\ref{Thm: main}, the second of Theorem~\ref{Thm: local} and that of Lemma~\ref{Lm: dirichlet}. The last in Theorem~\ref{Thm: global} is from the third one of Theorem~\ref{Thm: main}. We break it into two inequalities by setting $\dfrac{2n}{\rho_3K}\left(1+\ln\dfrac{2K^2\rho_3}{n}\right)\leq 1$.
\end{proof}
\begin{Rk}
The first restriction in Theorem~\ref{Thm: global} can be satisfied by making $\mu$ smaller while $\eps$ larger. This will lead to shorter stability time. The $\mu$ here plays the same role of the factor $g$ in \cite{LNN, N}. The second restriction can always be satisfied by making $\eps$ small. The third restriction can be satisfied by making $\mu$ or $\rho_1/\rho_3$ small. However, since $n!$ grows very fast, for large $n$, this restriction is easy to satisfy.\label{Rk: global}
\end{Rk}
\appendix
\section{Proof of Lemma~\ref{Lm: estimate}}\label{section: appendix}

The proof is done in the following Claim 1,2,3, which estimates $\Sigma_\pm,\Sigma_>,\Sigma_0$ in Lemma~\ref{Lm: estimate} respectively. Before the proof of the lemma, let us first analyze the geometry of numbers involved.

\subsection{The geometry of integer vectors}
Let us look at the Figure 2.  \\
\begin{figure}[ht]
\begin{center}
\includegraphics[width=0.6\textwidth]{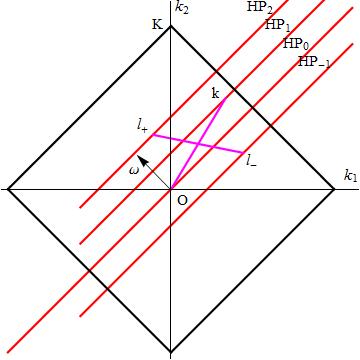}
\caption{counting the number of combinations.}
\end{center}\label{figure: no}
\end{figure}
\begin{itemize}
\item {\it The diamond}:  the diamond in the figure encloses all the vectors $k$ with $|k|\leq K$
(in 3-dim it is an octahedron. In general it is a ball of radius $K$ under the $l^1$ norm). The total number of integer
vectors inclosed in the diamond is $\dfrac{(2K)^n}{n!}$. Indeed, in $n$-dim, the diamond consists of $2^n$ simplices.
Each of the simplices has volume $\dfrac{K^n}{n!}$.
\item {\it The hyperplane}: the small arrow indicates the rational frequency $\om^*$. $\mathrm{HP}_0$ is a hyperplane that
is perpendicular to $\om^*$.  $\mathrm{HP}_0=D_0$ and the $(n-1)$-dim volume of $\mathrm{HP}_0\cap \mathrm{Diamond}$ is less than
$\dfrac{(2K)^{n-1}}{(n-1)!}$, which is the $(n-1)$-dim volume of an $(n-1)$-dim Diamond. Any vector lies above $\mathrm{HP}_0$ has positive inner product with $\om^*$,
while any vector below has negative inner products. Moreover, if two vectors lie on the same hyperplane which is parallel to $\mathrm{HP}_0$,
they will have the same inner product with $\om^*$. Let us denote $\mathrm{HP}_d=\{k\in\Z^n|\langle k,\om^*\rangle=d/\bar{T}\}$.   $\mathrm{HP}_0\cap \mathrm{Diamond}$ contains at most $\dfrac{(2K)^{n-1}}{(n-1)!}$ integer points.
\item {\it The parallelogram}: consider the vectors $l_+,l_-,k$ in the Figure 2. Suppose we have the
relation $l_++l_-=k$. Then the three vectors together with the origin form a parallelogram. Suppose $\langle k,\om^*\rangle \bar{T}=1$, and $\langle l_+\om^*\rangle\bar{T}=2$, then $\langle l_-,\om^*\rangle \bar{T}=-1$. $l_+$ and $l_-$ can move on their corresponding hyperplane, but a parallelogram is always preserved.
\item {\it The shape of the diamond under the averaging flow}: in the definition of $D_{\pm}(\dt)$,
we have the restriction $|l_{\pm}|\rho_3+|\langle l_{\pm},\om^*\rangle |\dt\leq \rho_3 K$ for $l_{\pm}\in D_{\pm}(\dt)$.
When $\dt=0$, this is our diamond. When $\dt$ increases, The diamond will collapse, i.e. the integer vectors becomes fewer on HP$_d$.
\begin{equation}|l_{\pm}|\leq  K-\dfrac{\dt}{\rho_3}|\langle l_{\pm},\om^*\rangle |.\label{eq: collapse}\end{equation}
The rate of decreasing depends on the inner product $|\langle l_{\pm},\om^*\rangle |$. The farther a hyperplane $\mathrm{HP}_d$ is away from $\mathrm{HP}_0$
(The larger the $d$), the faster it collapses (with volume decreasing rate $d/(\rho_3\bar{T})$). $\mathrm{HP}_0$ does not change at all.
When $\dt=\dt^*$, the diamond would collapse to its intersection with $\mathrm{HP}_0$. By then we would have
successfully killed all the nonresonant terms up to the desired exponential smallness $e^{-\rho_2|k|-\rho_3 K}$.
We denote the collapsed diamond at time $\dt$ by $\mathrm{Diamond}(\dt)$.
\end{itemize}
\subsection{Estimate of $\Sm_\pm, \Sm_>,\Sm_0$ and the proof of Lemma~\ref{Lm: estimate}}\label{subsection: estimate}
Now we obtain estimates of $\Sm_\pm, \Sm_>,\Sm_0$ for fixed $k$.
\subsubsection{Claim 1:}
\emph{The sum $\Sm_\pm(\dt)$ defined in equation~$(\ref{eq: sum})$ can be estimated as follows: }
\begin{equation}
\Sm_\pm(\dt)\leq
\begin{cases}
\dfrac{(2K)^{n}}{2n!} & \ \textrm{if} \ \dt\leq \dfrac{\bar{T}n}{2K},\\
\dfrac{(2K)^{n-1}\bar{T}}{2(n-1)!\dt} & \ \textrm{if}\  \dt\geq \dfrac{\bar{T}n}{2K}.
\end{cases}\nonumber
\end{equation}

\begin{proof} In the proof, the vector $k$ is fixed.
The $\Sm_\pm$ is defined in equation~(\ref{eq: sum}),
\[\Sm_\pm(\dt)=\sum_{l_{+}+l_-=k}e^{-(|l_{+}|+|l_-|-|k|)(\rho_3+2\rho_2)}e^{-\left (S_{l_{+}}\langle\om^*,l_{+}\rangle +S_{l_-}\langle\om^*,l_-\rangle -S_{k}\langle\om^*,k\rangle\right)},\]
which can be estimated as\[\Sm_\pm(\dt)\leq \displaystyle\sum_{l_++l_-=k}e^{-(S_{l_{+}}\langle\om^*,l_{+}\rangle +S_{l_-}\langle\om^*,l_-\rangle -S_{k}\langle\om^*,k\rangle)},\]
where we have dropped $e^{-(\rho_3+2\rho_2)(|l_+|+|l_-|-|k|)}$, since
\[l_++l_-=k\Longrightarrow|l_+|+|l_-|\geq|k|\Longrightarrow e^{-(\rho_3+2\rho_2)(|l_+|+|l_-|-|k|)}\leq 1.\]
From the relation $\langle\om^*,l_{+}\rangle +\langle\om^*,l_-\rangle =\langle\om^*,k\rangle$, and the inequality
$\dt\geq |S_k(\dt)|$, we have the following two cases depending on the sign of $\langle \om^*,k\rangle$.\\
If $\langle \om^*,k\rangle\geq 0$, then
\[S_{l_{+}}\langle\om^*,l_{+}\rangle +S_{l_-}\langle\om^*,l_-\rangle -S_{k}\langle\om^*,k\rangle\geq
\dt(\langle\om^*,k\rangle-\langle\om^*,l_-\rangle)-\dt\langle\om^*,l_-\rangle-S_k\langle\om^*,k\rangle\]
\[\geq (\dt-S_k)\langle \om^*,k\rangle-2\dt\langle \om^*,l_-\rangle\geq 2\dt|\langle \om^*,l_-\rangle|.\]
If $\langle \om^*,k\rangle\leq 0$, then
\[S_{l_{+}}\langle\om^*,l_{+}\rangle +S_{l_-}\langle\om^*,l_-\rangle -S_{k}\langle\om^*,k\rangle\geq
\dt\langle\om^*,l_+\rangle-\dt(S_k\langle\om^*,k\rangle-\langle\om^*,l_+\rangle)-S_k\langle\om^*,k\rangle\]
\[\geq (-\dt-S_k)\langle \om^*,k\rangle+2\dt\langle \om^*,l_+\rangle\geq 2\dt|\langle \om^*,l_+\rangle|.\]

Moreover, when $\dt=0$, the number of integer vectors contained in $\mathrm{HP}_d\cap \mathrm{Diamond}(0)$ is no greater than $\dfrac{(2K)^{n-1}}{(n-1)!}$. It is zero when $|d|>K$.
Since $k$ is fixed, we can vary either $l_+$ or $l_-$. The other one will be determined uniquely. According to the analysis above,
we sum over $l_+$ if $\langle \om^*,k\rangle\geq 0$ while over $l_-$ otherwise. We consider the $l_+$ case for instance.
As $\dt$ increases, on each $\mathrm{HP}_d$ the number of integer vectors contained in
$\mathrm{HP}_d\cap \mathrm{Diamond}(\dt)$ is no greater than $\dfrac{(2K-2d\dt /\bar{T}\rho_3)^{n-1}}{(n-1)!}$ according to inequality~(\ref{eq: collapse}).
It is zero when $|d|>K$ or $K\leq d\dt/\bar{T}\rho_3$. Now we have the estimate
\begin{equation}
\begin{aligned}
\displaystyle
&\Sm_\pm(\dt)\leq \sum_d \dfrac{(2k-2d\dt/\bar{T}\rho_3)^{n-1}}{(n-1)!}e^{-2d\dt/\bar{T}}\leq\int_{0}^{K\wedge\frac{K\bar{T}\rho_3}{\dt}}\dfrac{(2K-2d\dt/\bar{T}\rho_3 )^{n-1}}{(n-1)!}e^{-2d\dt /\bar{T}}\ d(d)\\
&\leq \dfrac{(2K)^{n-1}}{(n-1)!}\int_{0}^{\infty}e^{-2d\dt /\bar{T}}\ d(d)\leq \dfrac{(2K)^{n-1}\bar{T}}{2(n-1)!\dt}. \\
\nonumber
\end{aligned}
\end{equation}
This  estimate is poor when $\dt$ is close to zero. But in fact when $\dt=0$, the upper bound is $\Sm_\pm(0)\leq\dfrac{(2K)^{n}}{2n!},  $ and the upper bound is monotonically decreasing w.r.t. $\dt$. So we can use the bound stated in Claim $1$. See figure 3.
\begin{figure}[ht]
\begin{center}
\includegraphics[width=0.6\textwidth]{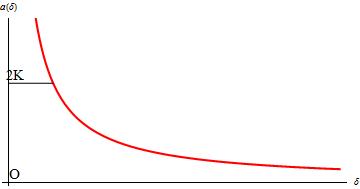}
\caption{upper bound of $a(\dt)$.}
\end{center}\label{figure: cut}
\end{figure}
\end{proof}
\subsubsection{Claim 2: }
\emph{The sum $\Sm_>(\dt)$ defined in equation~$(\ref{eq: sum})$ can be estimated as follows: }
$$
\displaystyle\Sm_>(\dt)\leq
\begin{cases}
\ \dfrac{(2K)^{n}}{n!} & if \ \dt\leq \dfrac{\bar{T}n}{2K},\\
\ \dfrac{(2K)^{n-1}\bar{T}}{(n-1)!\dt} & if \ \dt\geq \dfrac{\bar{T}n}{2K}.
\end{cases}
$$

\begin{proof}
It is defined in equation~(\ref{eq: sum}) that
\[\Sm_>(\dt)=\sum_{l_{\pm}+l_>=k}e^{-(|l_{\pm}|+|l_>|-|k|)(\rho_3+2\rho_2)}e^{-\left (S_{l_{\pm}}\langle\om^*,l_{\pm}\rangle +S_{l_>}\langle\om^*,l_>\rangle -S_{k}\langle\om^*,k\rangle\right)}\]
We first show \[\Sm_>(\dt)\leq \displaystyle\sum_{l_{\pm}+l_>=k}e^{-S_{l_{\pm}}\langle\om^*,l_{\pm}\rangle}.\]

According to the definition of $l_>\in D_>(\dt)$, the Fourier term corresponding to $l_>$ is of the size
\[e^{-S_{l_>}\langle\om^*,l_>\rangle -|l_>|\rho_3}\leq e^{-\rho_3K}.\]
But $e^{-S_{k}\langle\om^*,k\rangle -\rho_3|k|}\geq e^{-\rho_3K}$ due to Definition~\ref{Def: main def}.
The equality is achieved only when $k\in D_{>}(\dt)$. So we know \[e^{-S_{l_>}\langle\om^*,l_>\rangle -|l_>|\rho_3+ S_{k}\langle\om^*,k\rangle +\rho_3|k|}\leq 1.\]
We drop the $e^{-|l_+|\rho_3}$ and  $e^{-(|l_{\pm}|+|l_>|-|k|)(2\rho_2)}$ to obtain $\Sm_>(\dt)\leq \displaystyle\sum_{l_{\pm}+l_>=k}e^{-S_{l_{\pm}}\langle\om^*,l_{\pm}\rangle}$.
This is essentially the same as the Case $\Sm_\pm$ discussed above. So we have the bound stated in Claim $2$.\\
\end{proof}
\subsubsection{Claim $3$: }
\emph{The sum $\Sm_0(\dt)$ defined in equation~(\ref{eq: sum}) can be estimated as follows: }
\[\Sm_0(\dt)\leq \dfrac{(2K-2|\langle\om^*,k\rangle|\dt/\rho_3)^{n-1}}{(n-1)!}\leq\dfrac{(2K-2\dt/\bar{T}\rho_3)^{n-1}}{(n-1)!}.\]

\begin{proof} The term $\displaystyle\Sm_0(\dt)$ turns out to be the most troublesome term. Again equation~(\ref{eq: sum}) tells us
\[\Sm_0(\dt)=\sum_{l_{\pm}+l_0=k}e^{-(|l_{\pm}|+|l_0|-|k|)(\rho_3+2\rho_2)}e^{-\left (S_{l_{\pm}}\langle\om^*,l_{\pm}\rangle +S_{l_0}\langle\om^*,l_0\rangle -S_{k}\langle\om^*,k\rangle\right)}\]

Because of the relation $\langle\om^*,l_{\pm}\rangle +\langle\om^*,l_0\rangle =\langle\om^*,k\rangle$ and $\langle\om^*,l_0\rangle=0$, we get
$\langle\om^*,l_{\pm}\rangle =\langle\om^*,k\rangle$ and \[S_{l_{\pm}}(\dt)\langle\om^*,l_{\pm}\rangle\geq S_{k}(\dt)\langle\om^*,k\rangle.\]
The reason is, we know that $l_\pm\in D_\pm(\dt)$ until time $\dt$, but we do not know where $k$ is. This implies $|S_{l_\pm}(\dt)|\geq |S_{k}(\dt)|$. The ``=" is achieved only if $k\in D_-(\dt)\cup D_+(\dt)$.\\
So we get \[\displaystyle\Sm_0(\dt)\leq \sum_{l_{\pm}+l_0=k}e^{-(|l_{\pm}|+|l_0|-|k|)(\rho_3+2\rho_2)}\leq \sum_{l_{\pm}+l_0=k} 1.\] Now consider Figure 2.
Since $l_0\in \mathrm{HP}_0$,  we get $l_{\pm}$ and $k$ must lie on the same hyperplane.
So $\sum_{l_{\pm}+l_0=k} 1$ is bounded by the number of the possible $\l_{\pm}$'s, which is
\[\dfrac{(2K-2|\langle\om^*,k\rangle|\dt/\rho_3)^{n-1}}{(n-1)!}.\]
This gives the Claim $3$.\\
\end{proof}
\begin{proof}[Proof of Lemma~\ref{Lm: estimate}]
We simply add up the upper bounds for $\Sm_0,\Sm_>,\Sm_\pm$ to get an upper bound for $2\Sm_\pm+\Sm_>+\Sm_0$. This proves Lemma~\ref{Lm: estimate}.
\end{proof}
\section{Elements on majorant estimates}\label{section: majorant}
In this appendix, we collect some basics about the majorant relation. The materials can be found in the Chapter 5 of \cite{TZ}. The majorant relation $``\ll"$ is defined in Definition~\ref{Def: maj1}.
\begin{Lm}
The relation ``$\ll$" satisfies the following properties:
\begin{enumerate}
\item If $f_1 \ll g_1$ and $f_2 \ll g_2$, then $f_1 + f_2 \ll g_1 + g_2$ and $f_1f_2\ll g_1g_2$.
\item If $f \ll g$, then $\dfrac{\pt f}{\pt z_j} \ll \dfrac{\pt g}{\pt z_j}$ for any $j = 1,\cdots,m$.
\item If $f(z,\lambda)\ll g(z,\lambda)$ for any value of the parameter $\lambda\in[a, b]$, then
\[\int_a^b f(z,\lambda)\ d\lambda\ll \int_a^b g(z,\lambda)\ d\lambda\]
\item Let $|f(z)| \leq c$ in the domain
$\{z = (z_1,\cdots, z_m): |z|\leq b, j = 1,\cdots,m\}$.
Then $f(z)\ll c/w\ll \dfrac{bc}{b-z}$, where $w = b^{-m}(b-z_1)\cdots(b-z_m)$ and $z=z_1+z_2+\cdots+z_m.$
\end{enumerate}\label{Lm: digression}
\end{Lm}
Moreover, the majorant relation is also preserved by solving differential equations or integral equations.
\begin{Def}
Consider an ODE system
\begin{equation}
 f^k_\dt(z,\dt) = F^k(f(z,\dt), z,\dt),\quad  f^k(z,0) = \hat {f}^k(z)\label{eq: majorated}
 \end{equation}
with some known functional $F^k$ and initial data $\hat{f}^k$.
We call a system
 \begin{equation}
 \mathbf{f}^k_\dt(z,\dt)= \mathbf{F}^k(\mathbf{f}(z,\dt),z,\dt),\quad \mathbf{f}^k(z, 0)=\hat{\mathbf{f}}^k(z)\label{eq: majorating}
 \end{equation}
a majorant system associated with equation~$(\ref{eq: majorated})$ if\\
 $(a)$ $\hat{f}^k(z)\ll \hat{\mathbf{f}}^k(z)$ for any $k\in\Z$, and\\
 $(b)$ $F^k(g(z), z,\dt)\ll \mathbf{F}^k(\mathbf{g}(z), z,\dt)$ for any $k \in \Z$, $\dt\geq 0$, and $g, \mathbf{g}$ such that $g \ll \mathbf{g}$.\\
\label{Def: maj}
\end{Def}
We have the theorem
\begin{Thm}[Chapter 5 of \cite{TZ}]
If $\mathbf{f}(z,\dt), 0\leq\dt\leq\dt_0$ is a
solution of the majorant system~$(\ref{eq: majorating})$ associated with~$(\ref{eq: majorated})$, then the system~$(\ref{eq: majorated})$
has a solution and
$f^k(z,\dt)\ll \mathbf{f}^k(z,\dt)$ for any $\dt\in [0, \dt_0]$, $k\in \Z$.
The same is true if we rewrite systems~$(\ref{eq: majorated})~(\ref{eq: majorating})$ in the integral form:
\begin{equation}
\begin{aligned}
&f^k(z,\dt) = \hat{f}^k(z)+\int_0^\dt F^k(f(z,s),z,s)\ ds,\\
&\mathbf{f}^k(z,\dt)=\hat{\mathbf{f}}^k(z)+\int_0^\dt\mathbf{F}^k({\mathbf{f}(z,s)},z,s)\ ds,
\end{aligned}
\end{equation}
\label{Thm: Z}
\end{Thm}
With this theorem, we treat $\dt$ as a parameter instead of a variable. So we do not need to do the Taylor expansion w.r.t. $\dt$.

\section*{Acknowledgement}
The author would like to thank Prof. Vadim Kaloshin to provide the idea of applying continuous averaging to the Nekhoroshev theorem. He also would like to thank Prof. Treschev, Neishtadt, Dr. M. Guardia and A. Bounemoura for carefully reading the manuscript and giving many valuable suggestions. The last version of the paper is completed when the author is visiting IAS and he would like to thank IAS for its hospitality.

\end{document}